\documentclass{article}

\usepackage{arxiv}

\usepackage[utf8]{inputenc} 
\usepackage[T1]{fontenc}    
\usepackage[hypertexnames=true]{hyperref}       
\usepackage{url}            
\usepackage{booktabs}       
\usepackage{amsmath, amsfonts}       
\usepackage{amsthm}
\usepackage{nicefrac}       
\usepackage{microtype}      
\usepackage{cleveref}       
\usepackage{graphicx}
\usepackage{doi}

\usepackage{amssymb}
\usepackage{algorithmic}
\usepackage[ruled,linesnumbered,vlined]{algorithm2e}
\usepackage{array}
\usepackage[caption=false,font=normalsize,labelfont=sf,textfont=sf]{subfig}
\usepackage{textcomp}
\usepackage{stfloats}
\usepackage{verbatim}
\usepackage{cite}
\usepackage{multirow} 
\usepackage{scalerel}
\usepackage{tikz}
\usepackage[title,titletoc]{appendix}

\DeclareMathOperator*{\argmin}{\arg\min}

\theoremstyle{plain}

\newtheorem{corollary}{Corollary}
\newtheorem{lemma}{Lemma}
\newtheorem{proposition}{Proposition}
\theoremstyle{definition}
\newtheorem{assumption}{Assumption}
\newtheorem{definition}{Definition}
\theoremstyle{remark}
\newtheorem{remark}{Remark}
\newtheorem{example}{Example}

\title{Cardinality-Constrained Multi-Objective Optimization: Novel Optimality Conditions and Algorithms}

\author{ \href{https://orcid.org/0000-0002-2488-5486}{\includegraphics[scale=0.06]{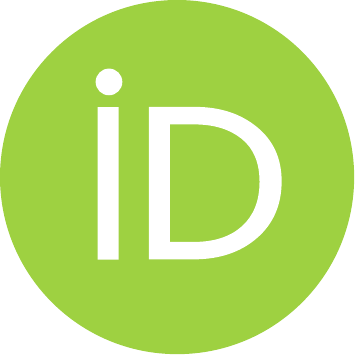}\hspace{1mm}Matteo Lapucci}\\
	Global Optimization Laboratory (GOL) \\
	Department of Information Engineering \\
	University of Florence \\
	Via di Santa Marta, 3, 50139, Florence, Italy \\
	\texttt{matteo.lapucci@unifi.it} \\
	\And
	\href{https://orcid.org/0000-0002-1394-0937}{\includegraphics[scale=0.06]{orcid.pdf}\hspace{1mm}Pierluigi Mansueto} \\
	Global Optimization Laboratory (GOL) \\
	Department of Information Engineering \\
	University of Florence \\
	Via di Santa Marta, 3, 50139, Florence, Italy \\
	\texttt{pierluigi.mansueto@unifi.it} \\
}

\renewcommand{\headeright}{Matteo Lapucci and Pierluigi Mansueto}
\renewcommand{\undertitle}{}
\renewcommand{\shorttitle}{}

\hypersetup{
pdftitle={Cardinality-Constrained Multi-Objective Optimization: Novel Optimality Conditions and Algorithms},
pdfsubject={},
pdfauthor={Matteo Lapucci, Pierluigi Mansueto},
pdfkeywords={Multi-objective optimization, Cardinality constraints, Optimality conditions, Numerical methods, Pareto front},
}

\begin{document}
\maketitle

\begin{abstract}
	In this paper, we consider multi-objective optimization problems with a sparsity constraint on the vector of variables. For this class of problems, inspired by the homonymous necessary optimality condition for sparse single-objective optimization, we define the concept of $L$-stationarity and we analyze its relationships with other existing conditions and Pareto optimality concepts. We then propose two novel algorithmic approaches: the first one is an Iterative Hard Thresholding method aiming to find a single $L$-stationary solution, while the second one is a two-stage algorithm designed to construct an approximation of the whole Pareto front. 
	Both methods are characterized by theoretical properties of convergence to points satisfying necessary conditions for Pareto optimality. Moreover, we report numerical results establishing the practical effectiveness of the proposed methodologies.
\end{abstract}

\MSCs{90C26 \and 90C29 \and 90C46}

\section{Introduction}

Multi-objective optimization (MOO) is a mathematical tool that has received loads of attention by the research community for the last 25 years. As a matter of fact, MOO problems turned out to be relevant in different application fields, where objectives that are in contrast with each other must be taken into account (see, e.g., \cite{Campana2018,Carrizosa1998,Fliege2001,pellegrini2014application,TAVANA20041147}).

With respect to single-objective optimization (SOO), the multi-objective case presents an additional complexity: in general, a solution minimizing all the objectives at once does not exist. This fact requires to introduce optimality conditions based on Pareto's theory, together with new, nontrivial, optimization schemes. Among the latter,  scalarization \cite{eichfelder09,Pascoletti1984} and evolutionary \cite{deb02,laumanns02} approaches are very popular. However, although these two algorithm classes have some appealing features, they also present important shortcomings. In particular, the scalarization approaches strongly depend on the problem domain and structure for the weight choices; moreover, even under strong regularity assumptions, they may lead to unbounded scalar problems for unfortunate weight selections \cite[section 7]{fliege09}. On the other hand, despite excelling at handling very difficult tasks, convergence properties are hard to prove for evolutionary approaches \cite{fliege09}. 

A different branch of MOO algorithms that is receiving increasing attention is that of the descent algorithms \cite{Cocchi2020,Drummond2004,fliege09,Fliege2000,Lapucci2023}. These methods basically extend the classic SOO descent methods. For most of them, theoretical convergence properties have been proved. The earliest developed algorithms of this class were only able to produce a single Pareto-stationary solution; in order to generate an approximation of the entire Pareto front, they were run multiple times in a multi-start fashion. More recently, some of these algorithms have been extended to overcome this limitation. These new approaches (see, e.g., \cite{COCCHI2021100008,Cocchi2020_Onthe,Cocchi2018,custodio11,lapucciimproved}) are capable of dealing with sequences of sets of points and thus producing a Pareto front approximation in an efficient and effective manner. Inspired by works for SOO \cite{LOCATELLI2014169,MANSUETO2021107849}, front-oriented descent methods have also been combined with genetic algorithms for MOO (see, e.g., \cite{Lapucci2022NSMA}).

A second topic recently investigated by the optimization community concerns problems where solutions with few nonzero components are required \cite{tillmann2021cardinality}. Solution sparsity is often induced by the direct introduction of a cardinality constraint on the variables vector. However, setting an upper-bound for the $\ell_0$ pseudo-norm makes the problem partly combinatorial and thus $\mathcal{NP}$-hard \cite{Bienstock1996,natarajan95}. For this reason, many approaches whose aim is to solve this problem approximately have been proposed. We refer the reader to \cite{tillmann2021cardinality} for a thorough survey of these methods. However, algorithms dealing exactly with the $\ell_0$ pseudo-norm can be found in the literature. In particular, the \textit{Iterative Hard Thresholding} (\texttt{IHT}) algorithm \cite{beck13}, the \textit{Penalty Decomposition} (\texttt{PD}) approach \cite{lu13} and the \textit{Sparse Neighborhood Search} (\texttt{SNS}) method \cite{lapucci2021unifying} are designed to be employable in the most general cases, even without convexity assumptions. With these methods, problems are tackled by means of continuous local optimization steps and convergence to solutions satisfying necessary optimality conditions is guaranteed. 

Although the two challenges have been thoroughly investigated separately, the combination of sparsity and multiple objectives has almost not been explored. The theoretical foundation for cardinality-constrained MOO was recently laid in \cite{lapucci22}, where a Penalty Decomposition approach was also proposed for sparse MOO tasks along with its convergence analysis. 
Moreover, a theoretical study extending the work from \cite{burdakov2016mathematical} to the MOO case was presented in \cite{garmanjani2022stationarity}.
The development of high-performing procedures to deal with this class of problems is beneficial for many real-world applications. For instance, there are several reasons in machine learning for inducing sparsity within classification/regression models (e.g., interpretability \cite{bertsimas2019price}, robustness \cite{weston2003use}, lightness \cite{carreira18}). In addition, there are approaches in the literature where learning tasks can be tackled from a Pareto-based, multi-objective perspective: fitting quality and model complexity are just two examples of conflicting objectives for which a good trade-off may be useful \cite{jin2008pareto}.

In this paper, we continue the theoretical analysis started in \cite{lapucci22}, introducing new optimality conditions for MOO problems with cardinality constraints. In particular, we define the concept of $L$-stationarity in MOO, which is directly inspired by the homonymous condition for sparse SOO tasks \cite{beck13}. Then, we introduce two new algorithms to solve these problems; the first one consists of an extension of the \texttt{IHT} method to the MOO case and it is designed to retrieve an $L$-stationary solution; we call this method \textit{Multi-Objective Iterative Hard Thresholding} (\texttt{MOIHT}) and prove that that it is indeed guaranteed to converge to points satisfying the newly introduced necessary optimality condition.
The second algorithm, on the other hand, is a two-stage approach whose ultimate goal is to approximate the whole Pareto front. This method, which we call \textit{Sparse Front Steepest Descent} (\texttt{SFSD}), is also analyzed theoretically and then shown, numerically, to actually reconstruct well the solution sets.

The remainder of the paper is organized as follows. In Section \ref{sec::preliminaries}, we first review the main MOO concepts along with some notions for cardinality-constrained MOO problems. In Section \ref{sec::optimality-conditions}, we define $L$-stationarity for the considered class of problems and state its theoretical relations with Pareto's theory and other existing optimality conditions. In Section \ref{sec::IHT_ext}, we provide a description of the \texttt{MOIHT} algorithm, along with its convergence analysis; moreover, in this section we propose the \texttt{SFSD} methodology. In Section \ref{sec::prel_ext}, we report the results of some computational experiments, showing the goodness of our novel approaches. Finally, in Section \ref{sec::conclusions}, we provide some concluding remarks.
\section{Preliminaries}
\label{sec::preliminaries}

In this paper, we consider problems of the following form:
\begin{equation}
	\label{eq::mo-prob}
		\min_{x \in \mathbb{R}^n}\;F\left(x\right)=\left(f_1\left(x\right),\ldots,f_m\left(x\right)\right)^\top
		\text{ s.t. }\left\|x\right\|_0 \le s,
\end{equation}
where $F: \mathbb{R}^n \rightarrow \mathbb{R}^m$ is a continuously differentiable vector-valued function, $\|\cdot\|_0$ denotes the $\ell_0$ pseudo-norm, i.e., the number of nonzero components of a vector, and $s \in \mathbb{N}$ is such that $1 \le s < n$.
In what follows, we indicate with $\|\cdot\|$ the Euclidean norm in $\mathbb{R}^n$.
We denote by $\Omega = \{x \in \mathbb{R}^n \mid \|x\|_0 \le s\}$ the closed and non-empty set induced by the upper bound on the $\ell_0$ pseudo-norm.

To deal with the multi-objective setting, we need to define a partial ordering in $\mathbb{R}^m$: given $u, v \in \mathbb{R}^m$, we say that $u < v$ if and only if, for all $j \in \{1,\ldots, m\}$, $u_j < v_j$; an analogous definition can be stated for the operators $\le, >, \ge$. Furthermore, given the objective function $F$, we say that $x$ \textit{dominates} $y$ w.r.t.\ $F$ if $F(x) \le F(y)$ and $F(x) \ne F(y)$ and we denote it by $F(x) \lneqq F(y)$. 

In multi-objective optimization, a solution which simultaneously minimizes all the objectives is unlikely to exist. In this case, optimality concepts are based on Pareto's theory.
\begin{definition}
	A point $\bar{x} \in \Omega$ is \textit{Pareto optimal} for problem \eqref{eq::mo-prob} if there does not exist any $y \in \Omega$ such that $F(y) \lneqq F(\bar{x})$. If there exists a neighborhood $\mathcal{N}(\bar{x})$ such that the property  holds in $\Omega\, \cap\, \mathcal{N}(\bar{x})$, then $\bar{x}$ is \textit{locally Pareto optimal}.
\end{definition}
Pareto optimality is a strong property and, as a consequence, it is often hard to achieve in practice. Then, a weaker but more affordable condition can be introduced.
\begin{definition}
	A point $\bar{x} \in \Omega$ is \textit{weakly Pareto optimal} for problem \eqref{eq::mo-prob} if there does not exist any $y \in \Omega$ such that $F(y) < F(\bar{x})$. If there exists a neighborhood $\mathcal{N}(\bar{x})$ such that the property  holds in $\Omega\, \cap\, \mathcal{N}(\bar{x})$, then $\bar{x}$ is \textit{locally weakly Pareto optimal}.
\end{definition}
We refer to the set of Pareto optimal solutions as the \textit{Pareto set}; the image of the latter under $F$ is called \textit{Pareto front}.

\medskip
\noindent\textbf{Additional notation:} Given an index set $S \subseteq \{1,\ldots, n\}$, the cardinality of $S$ is indicated with $|S|$, while we denote by $\bar{S} = \{1,\ldots, n\} \setminus S$ its \textit{complementary set}; we call $S$ a \textit{singleton} if $|S|= 1$. Letting $x \in \mathbb{R}^n$, we denote by $x_S$ the sub-vector of $x$ induced by $S$, i.e., the vector composed by the components $x_i$, with $i \in S$; $S_1(x) = \{i \in \{1,\ldots, n\} \mid x_i \ne 0\}$ represents the \textit{support set} of $x$, that is, the set of the indices corresponding to the non-zero components of $x$; $S_0(x) = \{1,\ldots, n\} \setminus S_1(x)$ is the $S_1(x)$ complementary set. Furthermore, according to \cite{beck16}, we say that an index set $J$ is a \textit{super support set} for $x$ if $S_1(x) \subseteq J$ and $|J|=s$; the set of all super support sets at $x$ is denoted by $\mathcal{J}(x)$ and it is a singleton if and only if $\|x\|_0=s$. Finally, we denote by $\boldsymbol{1}_N$ and $\boldsymbol{0}_N$, with $N \in \mathbb{N}^+$, the $N$-dimensional vectors of all ones and all zeros, respectively.

\subsection{The Proximal Operator in Multi-Objective Optimization}
Thorough analyses of proximal methods in the multi-objective setting can be found in the literature (see, e.g., \cite{bonnel05,Tanabe2019}). For the scope of this work, we refer to the discussion carried out in \cite{Tanabe2019}, where the considered MOO problems are of the form
\begin{equation}
	\label{eq:multi-prox-prob}
	\min_{x \in \mathbb{R}^n}\;\left(f_1\left(x\right) + g_1\left(x\right),\ldots,f_m\left(x\right) + g_m\left(x\right)\right)^\top.
\end{equation}
For all $j \in \{1,\ldots, m\}$, $f_j$ is assumed to be continuously differentiable, whereas $g_j$ is lower semi-continuous, proper convex but not necessarily smooth.

Let $x_k\in\mathbb{R}^n$. A \textit{proximal step} at $x_k$ can be carried out according to $x_{k + 1} = x_k + t_kd_k$, where $t_k$ is a suitable stepsize and the descent direction $d_k$ is obtained solving
\begin{equation}
	\label{eq::proximal-gradient-methods-problem}
	\min_{d \in \mathbb{R}^n}\;\max_{j \in \left\{1,\ldots, m\right\}}\left\{\nabla f_j(x_k)^\top d + g_j\left(x_k + d\right) - g_j\left(x_k\right)\right\} + \frac{L}{2}\left\|d\right\|^2,
\end{equation}
where $L > 0$.
An optimal solution of problem \eqref{eq:multi-prox-prob} is such that $\boldsymbol{0}_n$ is solution to \eqref{eq::proximal-gradient-methods-problem}.

Interestingly and similarly to the scalar case, problem \eqref{eq::proximal-gradient-methods-problem} can be seen as a generalization of well-known schemes to define the search direction: 
\begin{itemize}
	\item if, for all $j \in \{1,\ldots, m\}$, $g_j = 0$, then \eqref{eq::proximal-gradient-methods-problem} coincides with the problem of finding the steepest common descent direction \cite{Fliege2000};
	\item if, for all $j \in \{1,\ldots, m\}$, $g_j$ is the indicator function of a convex set $C$, then \eqref{eq::proximal-gradient-methods-problem} becomes equivalent to the projected gradient direction problem  \cite{Drummond2004}.
\end{itemize}

In the next section, we are going to show that the proximal operator can be used to handle the nonconvex set $\Omega$, in line with the work \cite{beck13} for scalar optimization. 

\section{Optimality Conditions}
\label{sec::optimality-conditions}

Under differentiability assumptions on the objective function $F$, a Pareto-stationarity condition was proved in \cite{lapucci22} to be necessary for (local) optimality. In what follows, we report slightly different definition and properties, adapted to problem \eqref{eq::mo-prob}.
\begin{definition}[{\cite[Definition 3.2]{lapucci22}}]
	\label{def::pareto_stat}
	A point $\bar{x} \in \Omega$ is Pareto-stationary for \eqref{eq::mo-prob} if 
	\begin{equation}
		\label{eq::basic_feasibility_MOO}
		\theta\left(\bar{x}\right) = \min_{d \in \mathcal{D}\left(\bar{x}\right)} \max_{j=1,\ldots,m} \nabla f_j(\bar{x})^\top d + \frac{1}{2} \left\|d\right\|^2 = 0,
	\end{equation}
	where $\mathcal{D}(\bar{x}) = \{v \in \mathbb{R}^n \mid \exists \bar{t} > 0: \bar{x} + tv \in \Omega\; \forall t \in [0, \bar{t}\,]\,\}= \{v \in \mathbb{R}^n \mid \|v_{S_0(\bar{x})}\|_0 \le s - \|\bar{x}\|_0\}$ is the set of feasible directions at $\bar{x}$.
\end{definition}
We denote by $v(\bar{x})$ the set of optimal solutions of problem \eqref{eq::basic_feasibility_MOO} at $\bar{x}$.
\begin{lemma}[{\cite[Proposition 3.3]{lapucci22}}]
	\label{lem::BF_necessary_for_Pareto}
	Let $\bar{x} \in \Omega$ be locally weakly Pareto optimal for problem \eqref{eq::mo-prob}. Then, $\bar{x}$ is Pareto-stationary for \eqref{eq::mo-prob}.
\end{lemma}
The second lemma states that, assuming the convexity of the objective functions, the stationarity condition is also sufficient for local weak Pareto optimality.
\begin{lemma}
	\label{lem::sufficientB}
	Assume $F$ is component-wise convex. Let $\bar{x} \in \Omega$ a Pareto-stationary point for problem \eqref{eq::mo-prob}. Then, $\bar{x}$ is locally weakly Pareto optimal for \eqref{eq::mo-prob}.
\end{lemma}
\begin{proof}
	See Appendix \ref{app::proof_basic_results}.
\end{proof}

Moreover, in \cite{lapucci22}, the Lu-Zhang first-order optimality conditions for scalar cardinality-constrained problems \cite{lu13} have been extended to the multi-objective optimization setting.

\begin{definition}[{\cite[Definition 3.6]{lapucci22}}]
	A point $\bar{x} \in \Omega$ satisfies the \textit{Multi-Objective Lu-Zhang first-order optimality conditions} (MOLZ conditions) for \eqref{eq::mo-prob} if there exists a super support set $J\in\mathcal{J}(\bar{x})$ such that
	\begin{equation}
		\label{eq::lu_zhang}
		\theta_J\left(\bar{x}\right) = \min_{d\in\mathbb{R}^n} \max_{j=1,\ldots,m} \nabla f_j(\bar{x})^\top d + \frac{1}{2}\left\|d\right\|^2 = 0 \quad \text{ s.t. }\quad d_{\bar{J}} = \boldsymbol{0}_{|\bar{J}|}
	\end{equation}
	
\end{definition}
Since problem \eqref{eq::lu_zhang} has a strongly convex objective function and a convex feasible set, it has a unique optimal solution at $\bar{x}$ that we indicate with $d_J(\bar{x})$.

\begin{lemma}[{\cite[Proposition 3.7]{lapucci22}}]
	\label{lem::MOLZ_necessary_for_BF}
	Let $\bar{x} \in \Omega$ be a Pareto stationary point for problem \eqref{eq::mo-prob}.Then, $\bar{x}$ satisfies MOLZ conditions.
\end{lemma}
As pointed out in \cite{lapucci22}, the converse is not always true; in order to obtain an equivalence between the two conditions, we need a stronger requirement.
\begin{lemma}[{\cite[Proposition 3.10]{lapucci22}}]
	A point $\bar{x} \in \Omega$ is a Pareto stationary point for problem \eqref{eq::mo-prob} if and only if it satisfies MOLZ conditions for all $J\in\mathcal{J}(\bar{x})$.
\end{lemma}

The Pareto-stationarity condition can be interpreted as a direct extension of the \textit{basic feasibility} concept in cardinality-constrained SOO \cite{beck13,beck16}. As such, the limitations of scalar basic feasibility naturally get transferred to the MOO case; in particular, Pareto-stationarity is only a local optimality condition and it does not allow to obtain information about the quality of the current support set. The MOLZ conditions emphasize this issue even more, being generally less restrictive than Pareto-stationarity.

With the above consideration in mind, we are motivated to extend the stronger $L$-stationarity condition from \cite{beck13} to the MOO case. In order to do so, we shall reinterpret $L$-stationarity in terms of proximal operators. Specifically, we can employ problem \eqref{eq::proximal-gradient-methods-problem} to define $L$-stationarity for MOO.

Let us consider the problem 
\begin{equation}
	\label{eq::MOO_L_stationarity}
	\theta_L\left(\bar{x}\right) = \min_{d \in  \mathcal{D}_L\left(\bar{x}\right)}\;\max_{j=1,\ldots,m} \nabla f_j(\bar{x})^\top d + \frac{L}{2}\left\|d\right\|^2,
\end{equation}  
where $\mathcal{D}_L(\bar{x}) = \{v \in \mathbb{R}^n \mid \bar{x} + v \in \Omega\}$ and let us denote by $v_L(\bar{x})$ the set of optimal solutions at $\bar{x}$ (since $\Omega$ is not a convex set, the solution is not necessarily unique).
It is easy to notice that problem \eqref{eq::MOO_L_stationarity} is equivalent to \eqref{eq::proximal-gradient-methods-problem} where, for all $j \in \{1,\ldots, m\}$, $g_j$ is the indicator function of the set $\Omega$. 
\begin{lemma}
	\label{lemma:continuity-theta}
	Let $\bar{x} \in \Omega$. Then, the following conditions hold:
	\begin{enumerate}
		\item $\theta_L(\bar{x})$ and $v_L(\bar{x})$ are well-defined;
		\item $\theta_L(\bar{x}) \le 0$;
		\item the mapping $\bar{x} \rightarrow \theta_L(\bar{x})$ is continuous.
	\end{enumerate}
\end{lemma}
\begin{proof}
	\begin{enumerate}
		\item The proof is trivial since
		\begin{itemize}
			\item the feasible set $ \mathcal{D}_L(\bar{x})$ is closed and non-empty,
			\item $\max_{j=1,\ldots,m} \nabla f_j(\bar{x})^\top d + \frac{L}{2}\|d\|^2$ is strongly convex and continuous in $ \mathcal{D}_L(\bar{x})$.
		\end{itemize}
		\item Given $\hat{d} = \boldsymbol{0}_n$, we have that
		$\max_{j=1,\ldots,m} \nabla f_j(\bar{x})^\top\hat{d} + \frac{L}{2}\|\hat{d}\|^2 = 0$. Since $\hat{d}\in\mathcal{D}_L(\bar{x})$, we get the thesis.
		\item The proof is identical to the one of Proposition 4 in \cite{Drummond2004}. The argument is not spoiled by the set $\mathcal{D}_L(\bar{x})$ being nonconvex. 
	\end{enumerate}
\end{proof}

We are now ready to introduce the definition of $L$-stationarity in MOO.

\begin{definition}
	\label{def:L-stat}
	A point $\bar{x} \in \Omega$ is \textit{$L$-stationary} for problem \eqref{eq::mo-prob} if $\theta_L(\bar{x}) = 0$.
\end{definition}

\begin{remark}
	By simple algebraic manipulations, the problem in \eqref{eq::MOO_L_stationarity} can be rewritten as
	\begin{equation}
		\label{eq::L_stationarity_MOO_z}
		\min_{z \in \Omega}\;\max_{j=1,\ldots,m} \nabla f_j(\bar{x})^\top\left(z - \bar{x}\right) + \frac{L}{2}\left\|z - \bar{x}\right\|^2.
	\end{equation}
	We can now observe that, if $m=1$, Definition \ref{def:L-stat} actually coincides with scalar $L$-stationarity. Indeed, exploiting \eqref{eq::L_stationarity_MOO_z}, $\theta_L(\bar{x})$ is equivalent to
	\begin{gather*}
		\min_{z\in\Omega}\nabla f(\bar{x})^\top(z-\bar{x})+\frac{L}{2}\|z-\bar{x}\|^2=\min_{z\in\Omega} \frac{L}{2}\|z-\bar{x}+\frac{1}{L}\nabla f(\bar{x})\|^2 - \frac{1}{2L}\|\nabla f(\bar{x})\|^2.
	\end{gather*}
    The minimum in the above problem is attained for $z^*\in\Pi_\Omega[\bar{x}-\frac{1}{L}\nabla f(\bar{x})]$, with $\Pi_\Omega$ being the (not unique) Euclidean projection onto the nonconvex set $\Omega$. We thus have that $\theta_L(\bar{x})=0$ if
    $\frac{L}{2}\|z^*-\bar{x}\|^2 + \nabla f(\bar{x})^\top(z^*-\bar{x}) =0,$
    which is satisfied if $\bar{x}\in \Pi_\Omega[\bar{x}-\frac{1}{L}\nabla f(\bar{x})]$, i.e., $\bar{x}$ is $L$-stationary according to \cite{beck13}.
\end{remark} 

In the rest of the section, we analyze the relations between $L$-stationarity, Pareto optimality, Pareto-stationarity and MOLZ conditions. 
We begin showing that, for any $L > 0$, each $L$-stationary point is Pareto-stationary.
\begin{proposition}
	\label{prop::basic_necessary_for_L}
	Let $\bar{x} \in \Omega$ be an $L$-stationary point for problem \eqref{eq::mo-prob} with $L > 0$. Then, $\bar{x}$ is Pareto-stationary for \eqref{eq::mo-prob}.
\end{proposition}
\begin{proof}
	By contradiction, we assume that $\bar{x}$ is not Pareto-stationary for \eqref{eq::mo-prob}, i.e., there exists $\hat{d} \in \mathcal{D}(\bar{x})$ such that 
	\begin{equation}
		\label{eq::objfun_B_barx}
		0 > \max_{j \in \left\{1,\ldots, m\right\}}\nabla f_j\left(\bar{x}\right)^\top\hat{d} + \frac{1}{2}\|\hat{d}\|^2 \ge \max_{j \in \left\{1,\ldots, m\right\}}\nabla f_j\left(\bar{x}\right)^\top\hat{d},
	\end{equation}
	where the second inequality is justified by the non-negativity of the norm operator.

	We now define the direction $\tilde{d}(t)= t\hat{d}$.
	Given the definition of $\mathcal{D}(\bar{x})$ and the feasibility of $\hat{d}$, we have there exists $\bar{t} > 0$ such that $\bar{x} + \tilde{d}(t) \in \Omega\;\forall t \in [0, \bar{t}\,]$.
	Thus, by definition of $\mathcal{D}_L(\bar{x})$, for all $t \in [0, \bar{t}\,]$, $\tilde{d}(t) \in \mathcal{D}_L(\bar{x}).$
	Let us define the function $\tilde{\theta}_L:\mathbb{R}^n\times\mathbb{R}^n\rightarrow\mathbb{R}$ as $\tilde{\theta}_L(x, d) = \max_{j=1,\ldots,m}\nabla f_j(x)^\top d + \frac{L}{2}\|d\|^2.$
	By \eqref{eq::MOO_L_stationarity}, it follows that $\theta_L(\bar{x}) = \tilde{\theta}_L(\bar{x}, \bar{d}^L)$, where $\bar{d}^L\in v_L(\bar{x})$, and also 
	\begin{equation}
		\label{eq::theta_tildetheta}
		\theta_L\left(\bar{x}\right) \le \tilde{\theta}_L\left(\bar{x}, d\right), \quad \forall\,d \in \mathcal{D}_L\left(\bar{x}\right).
	\end{equation}
	Combining the definitions of $\tilde{d}(t)$ and $\tilde{\theta}_L(x,d)$, we get that $\tilde{\theta}_L(\bar{x}, \tilde{d}(t)) = t\max_{j=1,\ldots,m}\nabla f_j(\bar{x})^\top\hat{d} + t^2\frac{L}{2}\|\hat{d}\|^2.$
	It is easy to see that $\tilde{\theta}_L(\bar{x}, \tilde{d}(t)) < 0$	if
	\begin{equation}
		\label{eq::interval_t}
		0 < t < -\frac{2}{L\|\hat{d}\|^2}\max_{j=1,\ldots,m}\nabla f_j\left(\bar{x}\right)^\top\hat{d},
	\end{equation}
	where the right-hand side is a positive quantity as $L>0$ and \eqref{eq::objfun_B_barx} holds.
	
	Then, taking into account the feasibility of $\tilde{d}(t)$, \eqref{eq::theta_tildetheta} and \eqref{eq::interval_t}, we can define a direction $\tilde{d}(\hat{t})$, with $\hat{t} \in \left(0, \min\{\bar{t}, -\frac{2}{L\|\hat{d}\|^2}\max_{j=1,\ldots,m}\nabla f_j(\bar{x})^\top\hat{d}\}\right),$
	so that $\tilde{d}(\hat{t}) \in \mathcal{D}_L(\bar{x})$ and $\theta_L(\bar{x}) \le \tilde{\theta}_L(\bar{x}, \tilde{d}(\hat{t})) < 0$. We finally get a contradiction since, by hypothesis, $\bar{x}$ is an $L$-stationary point for \eqref{eq::mo-prob}, i.e., $\theta_L(\bar{x}) = 0$.
\end{proof}

The last result also highlights the relation between $L$-stationarity and MOLZ conditions, which, as stated in Lemma \ref{lem::MOLZ_necessary_for_BF}, are necessary for Pareto stationarity. We formalize it in the next corollary.

\begin{corollary}
	Let $\bar{x} \in \Omega$ be an $L$-stationary point for problem \eqref{eq::mo-prob} with $L > 0$. Then, $\bar{x}$ satisfies the MOLZ conditions.
\end{corollary}

Given Proposition \ref{prop::basic_necessary_for_L} and Lemma \ref{lem::sufficientB}, we can also state that, under convexity assumptions for $F$, $L$-stationarity is a sufficient condition for local weak Pareto optimality.
\begin{corollary}
	Assume that $F$ is component-wise convex and let $\bar{x} \in \Omega$ be an $L$-stationary point for problem \eqref{eq::mo-prob} with $L > 0$. Then, $\bar{x}$ is locally weakly Pareto optimal for \eqref{eq::mo-prob}.
\end{corollary}

In order to continue the analysis, we need to introduce a couple of notions. The first one is an assumption similar to the one used for $L$-stationarity in \cite{beck13}, while the second one concerns an adaptation of the \textit{descent lemma} to MOO.

\begin{assumption}
	\label{ass::lipschitz}
	For all $j \in \{1,\ldots, m\}$, $\nabla f_j$ is Lipschitz-continuous over $\mathbb{R}^n$ with constant $L(f_j)$, i.e., $\|\nabla f_j(x) - \nabla f_j(y)\| \le L(f_j)\|x - y\|$ for all $x, y \in \mathbb{R}^n.$

\end{assumption}

In what follows, we indicate with $L(F) \in \mathbb{R}^m$ the vector of the Lipschitz constants, i.e., $L(F) = (L(f_1),\ldots, L(f_m))^\top$.

\begin{lemma}[{\cite[Proposition A.24]{bertsekas1999nonlinear}}]
	\label{lem::descent-lemma}
	Let $f_j$, $j =1,\ldots,m$, be a continuously differentiable function satisfying Assumption \ref{ass::lipschitz}. Then, for all $L \ge L(f_j)$ and any $x, d \in \mathbb{R}^n$, we have that
	$f_j(x + d) \le f_j(x) + \nabla f_j(x)^\top d + \frac{L}{2}\|d\|^2$.
\end{lemma}

We are ready to show that, for specific $L$ values, the $L$-stationarity condition is necessary for weak Pareto optimality.

\begin{proposition}
	\label{prop::necessary_L}
	Let Assumption \ref{ass::lipschitz} hold, $\bar{x} \in \Omega$ be a weakly Pareto optimal point for problem \eqref{eq::mo-prob} and $L > \max_{j=1,\ldots,m}L(f_j)$. Then, $\bar{x}$ is $L$-stationary for \eqref{eq::mo-prob}. Moreover, we have that $v_L(\bar{x}) = \{\boldsymbol{0}_n\}$, i.e., the set $v_L(\bar{x})$ is a singleton.
\end{proposition}
\begin{proof}
	By contradiction, let us assume that either $\bar{x}$ is not $L$-stationary for \eqref{eq::mo-prob} or $v_L(\bar{x}) \setminus \{\boldsymbol{0}_n\} \neq \emptyset$. Then, there exists a direction $\hat{d} \in \mathcal{D}_L(\bar{x})$ such that $\hat{d} \ne \boldsymbol{0}_n$ and 
	\begin{equation}
		\label{eq::Lstat-absurd}
		\max_{j \in \left\{1,\ldots, m\right\}}\nabla f_j(\bar{x})^\top\hat{d} + \frac{L}{2}\|\hat{d}\|^2 \le 0.
	\end{equation}
	By Lemma \ref{lem::descent-lemma}, we have that, for all $h \in \{1,\ldots, m\}$,
	\begin{equation}
		\label{eq::dl_applied}
		f_{h}\left(\bar{x} + \hat{d}\right) \le f_{h}(\bar{x}) + \nabla f_{h}(\bar{x})^\top\hat{d} + \frac{L\left(f_{h}\right)}{2}\|\hat{d}\|^2.
	\end{equation}
	From Equation \eqref{eq::Lstat-absurd}, we get that $\nabla f_{h}(\bar{x})^\top\hat{d} \le \max_{j \in \{1,\ldots, m\}}\nabla f_j(\bar{x})^\top\hat{d} \le -\frac{L}{2}\|\hat{d}\|^2,$
	where the first inequality comes from the definition of maximum operator. Recalling the hypothesis on $L$ and the non-negativity of the norm, we combine \eqref{eq::dl_applied} and the last result obtaining that
	\begin{gather*}
		f_{h}\left(\bar{x} + \hat{d}\right) \le f_{h}(\bar{x}) + \frac{L\left(f_{h}\right) - L}{2}\|\hat{d}\|^2 < f_{h}(\bar{x}) + \frac{L\left(f_{h}\right) - \max_{j \in \left\{1,\ldots, m\right\}}L\left(f_{j}\right)}{2}\|\hat{d}\|^2.
	\end{gather*}
	Thus, for all $h \in \{1,\ldots, m\}$ we have $f_{h}(\bar{x} + \hat{d}) - f_{h}(\bar{x}) < \frac{\|\hat{d}\|^2}{2}(L(f_{h}) - \max_{j \in \{1,\ldots, m\}}L(f_{j})) \le 0,$
	leading to the conclusion that we have found a point $\bar{x} + \hat{d} \in \Omega$ such that $F(\bar{x} + \hat{d}) < F(\bar{x})$. This is a contradiction since, by hypothesis, $\bar{x}$ is weakly Pareto optimal for \eqref{eq::mo-prob}. Thus, we get the thesis.
\end{proof}

The analysis on $L$-stationarity highlights how the choice of the $L$ value could be crucial: if $L$ is too small, $L$-stationarity might not be a necessary optimality condition; on the other hand, if $L$ gets too large, all the Pareto stationary points also become $L$-stationary. This behavior can be better noticed with an example.
\begin{figure}[h]
	\centering
	\includegraphics[width=0.675\textwidth]{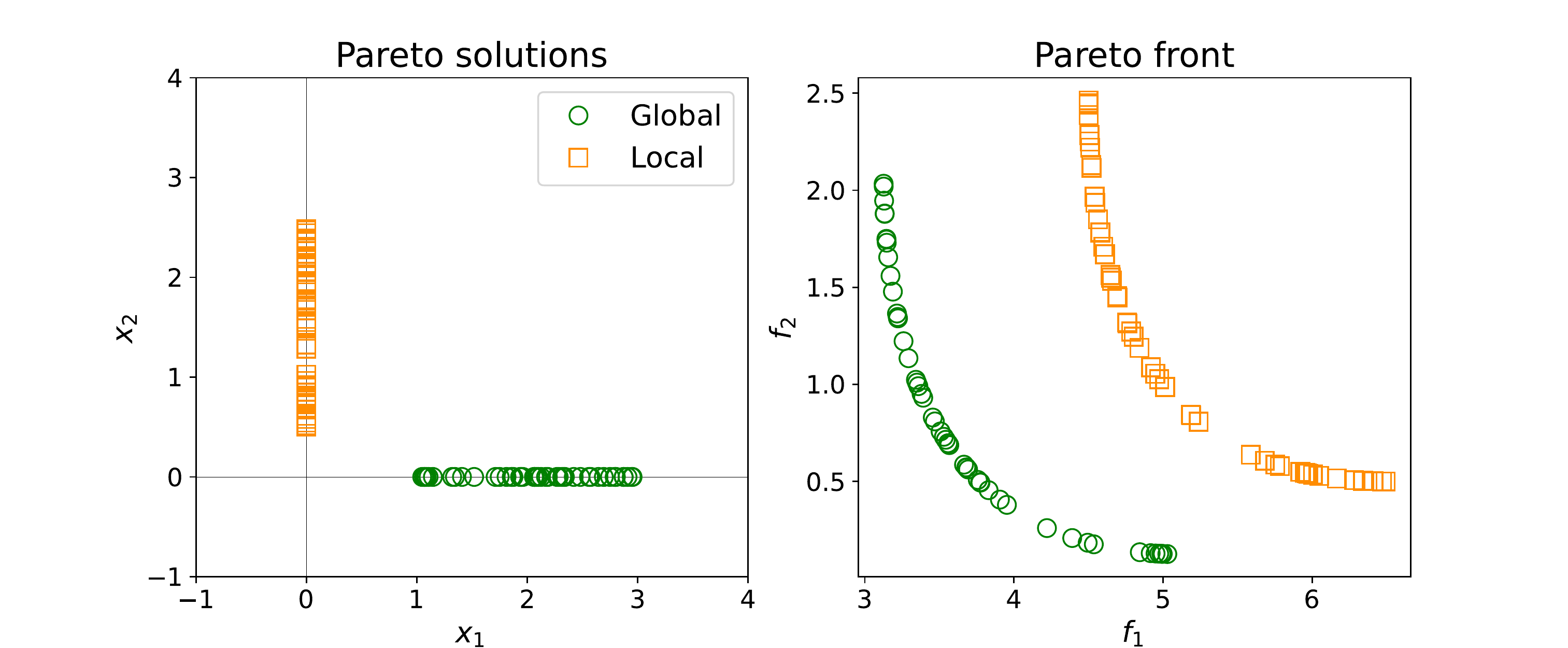}
	\caption{Pareto optimal solutions and Pareto front of problem of Example \ref{ex::L-stat}.}
	\label{fig::Example_problem}
\end{figure}
\begin{example}
	\label{ex::L-stat}
	Let us consider the following optimization problem:
	\begin{equation*}
		\begin{aligned}
			\min_{x\in \mathbb{R}^2}\quad \frac{1}{2}\left((x_1 - 3)^2 + (x_2 - 2.5)^2, (x_1 - 1)^2 + (x_2 - 0.5)^2\right)^\top\; \text{s.t.}\;\left\|x\right\|_0 \le 1.
		\end{aligned}
	\end{equation*}
	The Lipschitz constant of the gradient of both objective functions $f_j$ is $L(f_j) = 1$.
	In Figure \ref{fig::Example_problem}, the Pareto optimal solutions and the Pareto front are plotted: the problem has global optimal solutions corresponding to points with $x_1 \ne 0$; the local ones are characterized by the second component $x_2 \ne 0$. 
	By Lemmas \ref{lem::BF_necessary_for_Pareto}-\ref{lem::MOLZ_necessary_for_BF}, it follows that all the considered points are Pareto-stationary and satisfy the MOLZ conditions.
	In Figure \ref{fig::Example_problem_L}, we show which Pareto solutions are $L$-stationary, considering four different choices for $L$. If $L$ is chosen too small (Figure \ref{fig::L_0.75}), some global Pareto solutions do not result to be $L$-stationary. As stated in Proposition \ref{prop::necessary_L}, the $L$-stationarity condition turns out to be necessary for Pareto optimality for an $L$ value greater than the Lipschitz constants (Figure \ref{fig::L_1.01} where $L = 1.01$). On the other hand, a too high value makes the condition rather weak: in Figure \ref{fig::L_1.25} ($L = 1.25$), even some local Pareto solutions are $L$-stationary. The situation is further stressed in Figure \ref{fig::L_2.0} where $L = 2$ and all Pareto-stationary points are also $L$-stationary.
	\begin{figure}[h]
		\centering
		\subfloat[\label{fig::L_0.75}]{\includegraphics[width=0.375\textwidth]{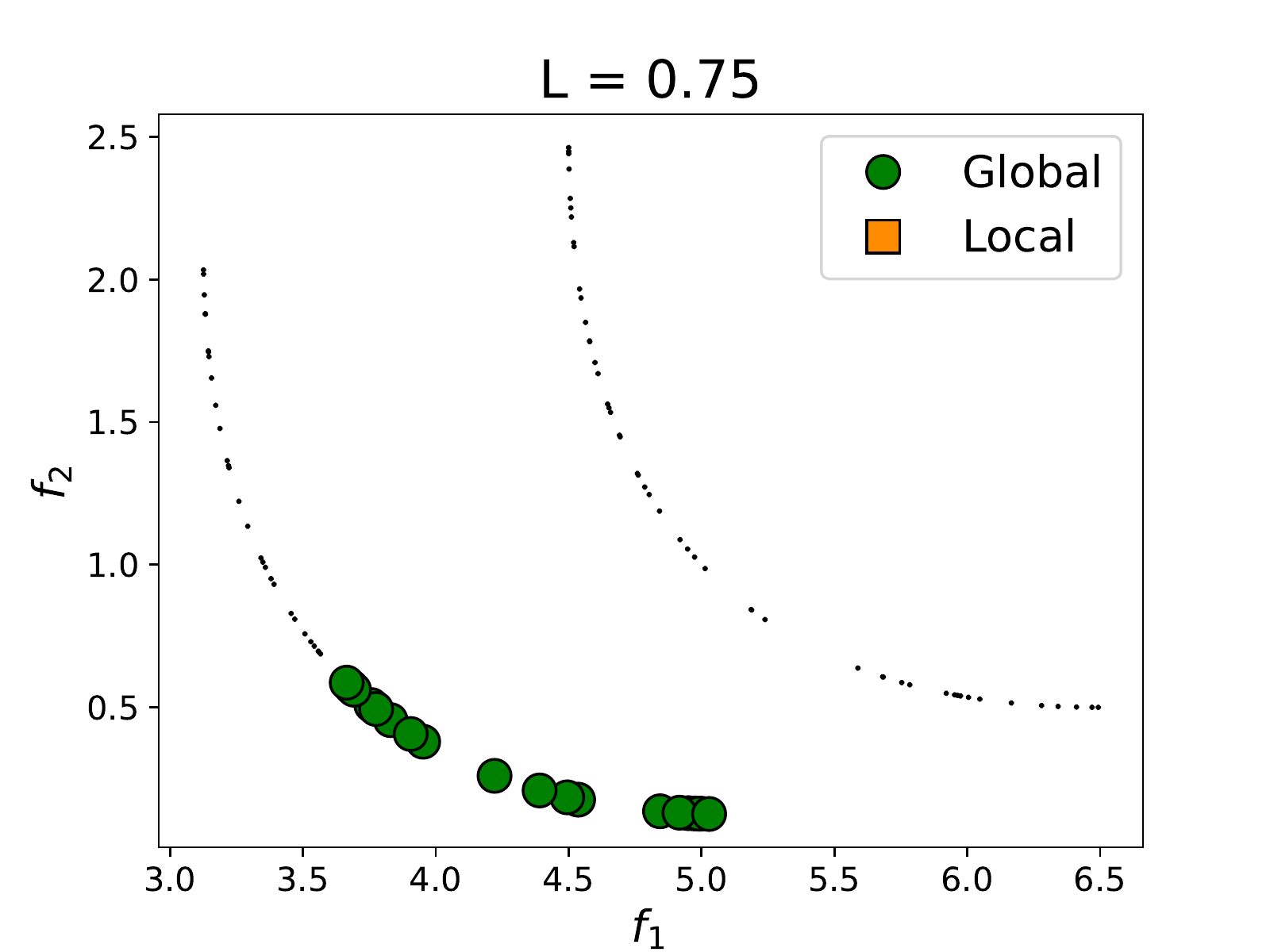}}
		\hfil
		\subfloat[\label{fig::L_1.01}]{\includegraphics[width=0.375\textwidth]{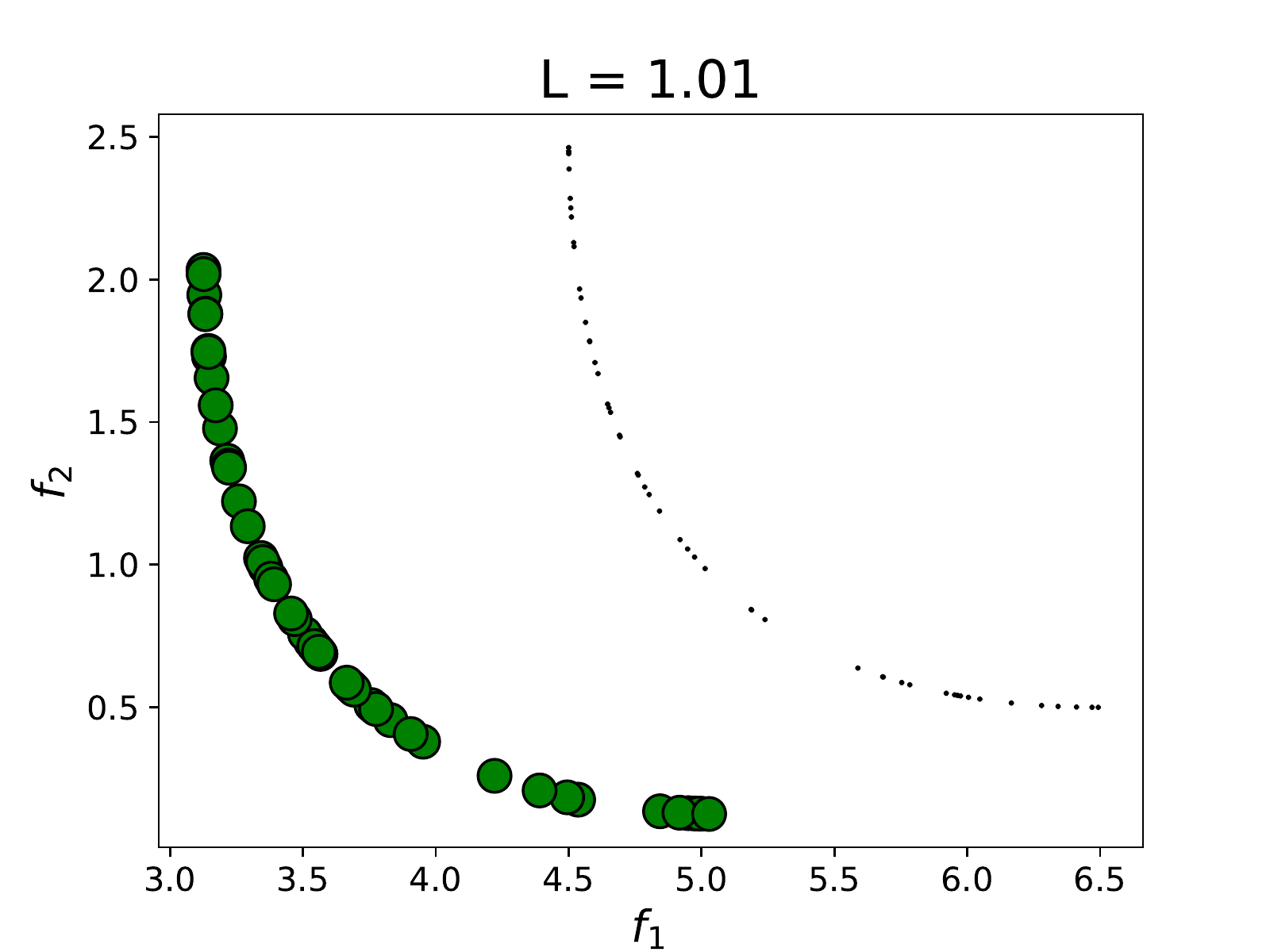}}
		\\
		\subfloat[\label{fig::L_1.25}]{\includegraphics[width=0.375\textwidth]{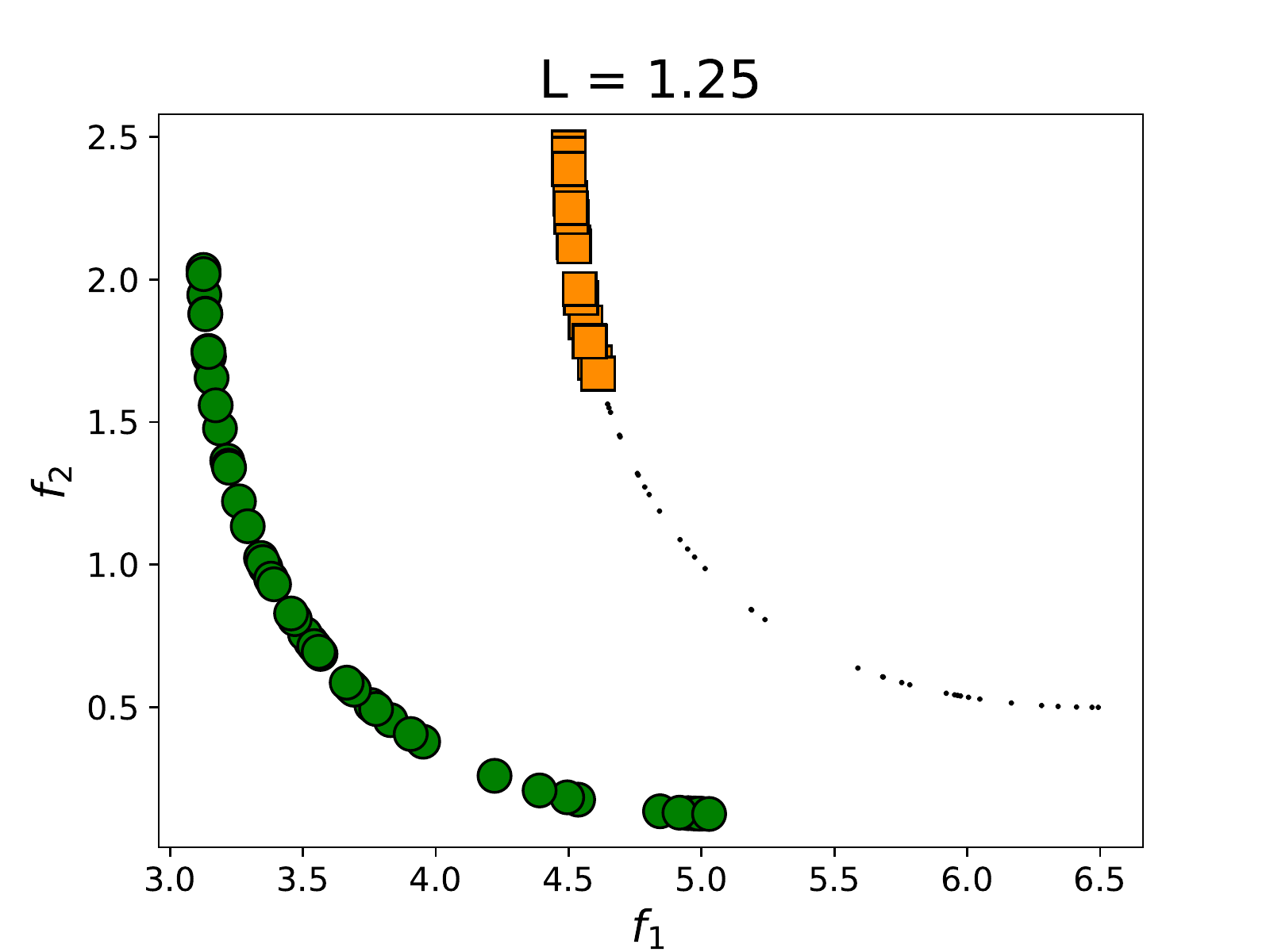}}
		\hfil
		\subfloat[\label{fig::L_2.0}]{\includegraphics[width=0.375\textwidth]{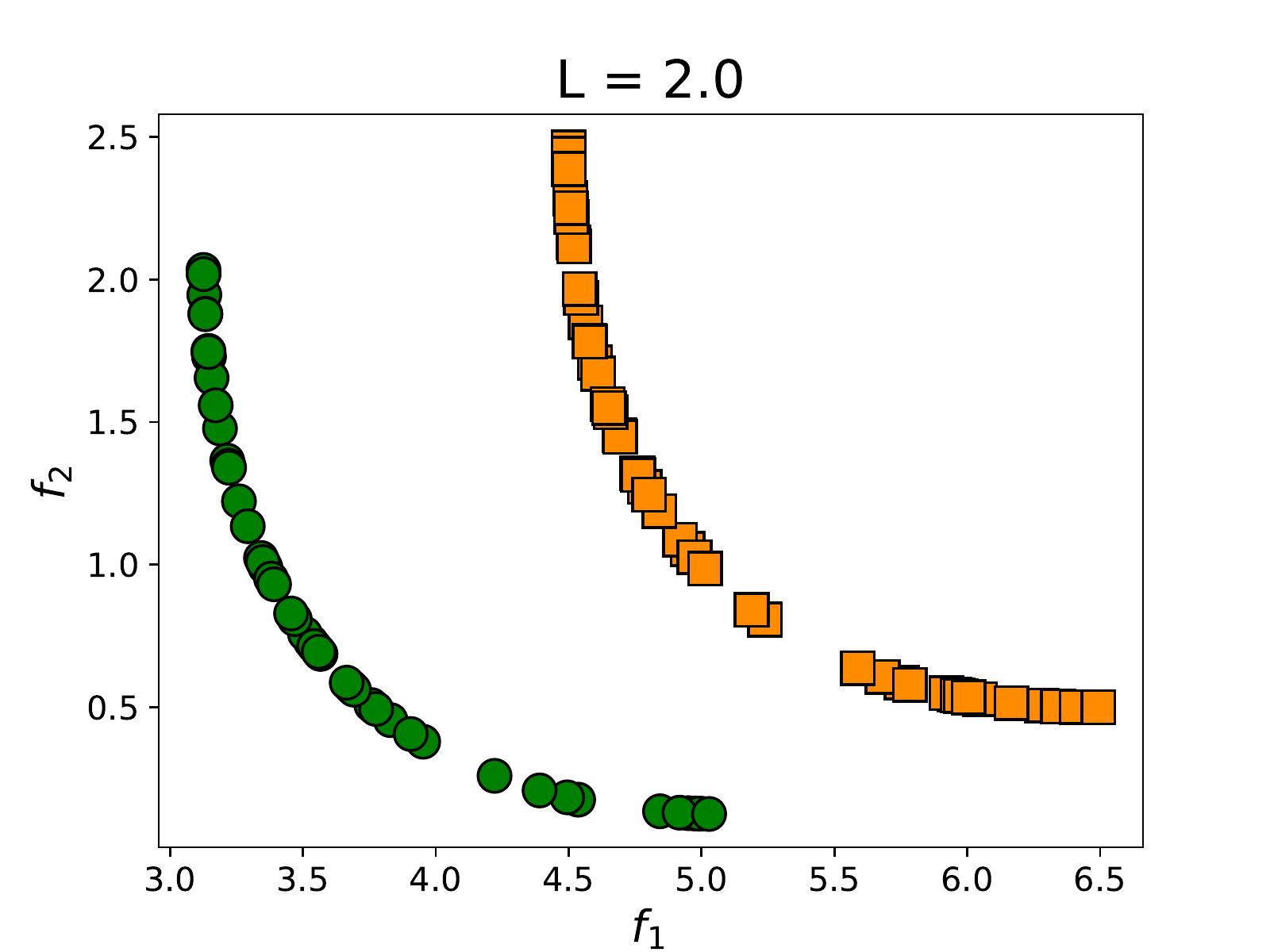}}
		\caption{$L$-stationary points in the Pareto front of problem of Example \ref{ex::L-stat} ($L(F) = [1, 1]^\top$) for different values of $L$.}
		\label{fig::Example_problem_L}
	\end{figure}
\end{example}

\section{New Algorithmic Approaches for Sparse MOO Problems}
\label{sec::IHT_ext}

In this section, we propose two procedures to solve cardinality-constrained MOO problems. The first one can be seen as an extension of the \textit{Iterative Hard Thresholding} (\texttt{IHT}) algorithm \cite{beck13}; the second one is a front approximation approach that takes as input candidate solutions, possibly associated with different support sets, and then spans the portions of the Pareto front associated with those supports. We report their schemes and discuss their properties in separate sections.

\subsection{Multi-Objective Iterative Hard Thresholding}
\label{subsec::MOIHT}
The first procedure we introduce is the \textit{Multi-Objective Iterative Hard Thresholding} algorithm. In the remainder of the paper, we refer to it as \texttt{MOIHT}.
The scheme of the method is reported in Algorithm \ref{alg::MO_IHT}.
At each iteration of \texttt{MOIHT}, the current solution $x_k$ is updated solving problem \eqref{eq::MO-IHT-problem}. The execution continues until an $L$-stationary point for \eqref{eq::mo-prob} is found. 
\begin{remark}
	\label{rem::feasibility-MO-IHT}
	At each iteration $k$, the solution $x_{k + 1}$ generated by \texttt{MOIHT} is feasible for \eqref{eq::mo-prob}. Indeed, the feasibility easily follows by definition of $\mathcal{D}_L(x_k)$.
\end{remark}
\begin{remark}
	It is very important to underline that Step \ref{step:update} is a practical operation that can be effectively implemented in the general case. Problem \eqref{eq::MO-IHT-problem} can indeed be solved up to global optimality, for example with mixed-integer programming techniques (see, e.g., \cite{bertsimas16,bertsimas2020sparse}). Defining a sufficiently large scalar $M>0$, the problem can be equivalently reformulated as
	\begin{align*}
			\min_{t,d,\delta}\;t + \frac{L}{2}\|d\|^2
			\quad\text{ s.t. }\quad &\nabla f_j(\bar{x})^\top d \le t\quad\forall j=1,\ldots,m,\qquad
			 \boldsymbol{1}_n^\top\delta \le s,\\ &
			 -M\delta \le \bar{x} + d \le M\delta,\qquad {t \in \mathbb{R},\, d \in \mathbb{R}^n,\,\delta \in \{0, 1\}^n}.
	\end{align*}
\end{remark}
\begin{algorithm}[htbp]
	\caption{\texttt{Multi-Objective Iterative Hard Thresholding}} \label{alg::MO_IHT}
	Input: $x_0 \in \Omega$, $L > \max_{j \in \{1,\ldots, m\}}L(f_j)$.\\
	$k = 0$ \\
	\While{$x_k$ is not $L$-stationary for problem \eqref{eq::mo-prob}}{
		Compute
		\begin{equation}
			\label{eq::MO-IHT-problem}
			v_L\left(x_k\right) = \argmin_{d \in \mathcal{D}_L\left(x_k\right)}\max_{j \in \left\{1,\ldots, m\right\}} \nabla f_j(x_k)^\top d + \frac{L}{2}\left\|d\right\|^2
		\end{equation} \label{step:update}\\
		Let $d_k^L \in v_L(x_k)$\\
		$x_{k + 1} = x_k + d_k^L$ \label{line::new_xk+1}\\
		Let $k = k + 1$
	}
	\Return $x_k$
\end{algorithm}

\subsubsection{Convergence Analysis}

In this section, we analyze the method from a theoretical perspective. Before proving the main convergence result, we need to state an additional assumption on $F$ and prove a technical lemma.
\begin{assumption}
	\label{ass::bounded-level-set}
	$F$ has bounded level sets in the multi-objective sense, i.e., the set $\mathcal{L}_F(z) = \{x \in \Omega \mid F(x) \le z\}$ is bounded for any $z \in \mathbb{R}^m$.
\end{assumption}
\begin{lemma}
	\label{lem::technical-lemma-MO-IHT}
	Let Assumptions \ref{ass::lipschitz}-\ref{ass::bounded-level-set} hold and $\{x_k\}$ be the sequence generated by Algorithm \ref{alg::MO_IHT} with constant $L > \max_{j \in \{1,\ldots, m\}}L(f_j)$. Then:
	\begin{enumerate}
		\item for all $k$, $F(x_k) - F(x_{k + 1}) \ge \frac{1}{2}\|x_k - x_{k + 1}\|^2(\boldsymbol{1}_mL - L(F));$
		\item for all $k$, if $x_k \ne x_{k + 1}$, then $F(x_{k + 1}) < F(x_k)$;
		\item for all $j \in \{1,\ldots, m\}$, the sequence $\{f_j(x_k)\}$ is non-increasing;
		\item the sequence $\{F(x_k)\}$ converges;
		\item $\lim_{k \rightarrow \infty}\|x_k - x_{k + 1}\|^2 = 0$.
	\end{enumerate}
\end{lemma} 
\begin{proof}
	\begin{enumerate}
		\item The thesis can be proved making an argument similar to that of Proposition \ref{prop::necessary_L} and reminding that $x_{k + 1} = x_k + d_k^L$, with $d_k^L \in v_L(x_k)$ (Step \ref{line::new_xk+1} of Algorithm \ref{alg::MO_IHT}).
		\item It follows directly from Point 1., recalling that $L>L(f_j)$ for all $j \in \{1,\ldots, m\}$ and $\|x_k-x_{k+1}\|>0$.
		\item By Point 1. and the hypothesis on $L$, we have that, for all $k$ and $j \in \{1,\ldots, m\}$, $f_j(x_{k + 1}) \le f_j(x_k)$. Thus, for all $j \in \{1,\ldots, m\}$, the sequence $\{f_j(x_k)\}$ is non-increasing.
		\item It follows from Assumption \ref{ass::bounded-level-set} and Point 3..
		\item From Point 1., we have that, for all $k$ and $j \in \{1,\ldots, m\}$,
		$\frac{L - L(f_{j})}{2}\|x_k - x_{k + 1}\|^2 \le f_{j}(x_k) - f_{j}(x_{k + 1}).$
		Since $f_{j}$ is continuous, we can take the limit for $k \rightarrow \infty$ on both sides of the inequality:
		$\lim_{k \rightarrow \infty}\frac{L - L(f_{j})}{2}\|x_k - x_{k + 1}\|^2 \le \lim_{k \rightarrow \infty}f_{j}(x_k) - f_{j}(x_{k + 1}) = 0,$
		where the equality comes from Point 4.. By the definition of $L$ and the non-negativity of the norm, the statement is proved.
	\end{enumerate}
\end{proof}
\begin{proposition}
	\label{prop:moiht-conv}
	Let Assumptions \ref{ass::lipschitz}-\ref{ass::bounded-level-set} hold and $\{x_k\}$ be the sequence generated by Algorithm \ref{alg::MO_IHT} with constant $L > \max_{j \in \{1,\ldots, m\}}L(f_j)$. Then, the sequence admits cluster points, each one being $L$-stationary for problem \eqref{eq::mo-prob}.
\end{proposition}
\begin{proof}
	First, we prove that the sequence admits limit points. By Lemma \ref{lem::technical-lemma-MO-IHT}, we deduce that, for all $k$,
	$F(x_k) \le F(x_{k - 1}) \le \ldots \le F(x_0).$
	Moreover, as noted in Remark \ref{rem::feasibility-MO-IHT}, $x_k \in \Omega$ for all $k$. Thus, we have that
	$x_k \in \mathcal{L}_F(F(x_0))\; \forall k.$
	Since Assumption \ref{ass::bounded-level-set} holds, we conclude that the sequence $\{x_k\}$ is bounded and it thus admits limit points.
	Now, by contradiction, let us suppose there exists a subsequence $K \subseteq \{0,1,\ldots\}$ such that 
	\begin{equation}
		\label{eq::lim_barx}
		\lim_{\substack{k \rightarrow \infty \\ k \in K}}x_k = \bar{x}
	\end{equation}
	and $\bar{x}$ is not $L$-stationary for problem \eqref{eq::mo-prob}. Then, there exists $\hat{\varepsilon} > 0$ such that $\theta_L(x_k) \le -\hat{\varepsilon} < 0$ for all $k \in K$.
	By the instructions of the algorithm, at each iteration $k$, $d_k^L \in v_L(x_k)$ solves problem \eqref{eq::MO-IHT-problem}. Moreover, by Step \ref{line::new_xk+1}, $d_k^L = x_{k+1} - x_k$. Thus,
	$
		\max_{j \in \{1,\ldots, m\}}\nabla f_j(x_k)^\top(x_{k + 1} - x_k) + \frac{L}{2}\|x_{k + 1} - x_k\|^2 = \theta_L(x_k) \le -\hat{\varepsilon},
	$
	which can be rewritten as
	\begin{equation}
		\label{eq::max_varepsilon_L}
		\max_{j \in \left\{1,\ldots, m\right\}}\nabla f_j(x_k)^\top\left(x_{k + 1} - x_k\right) \le -\hat{\varepsilon} - \frac{L}{2}\left\|x_{k + 1} - x_k\right\|^2.
	\end{equation}
	By Lemma \ref{lem::descent-lemma} and Step \ref{line::new_xk+1} of the algorithm, we have that, given $h \in \{1,\ldots, m\}$,
	\begin{equation}
		\label{eq::descent_lemma_applied_MO-IHT}
		f_{h}(x_{k + 1}) \le f_{h}(x_k) + \nabla f_{h}(x_k)^\top\left(x_{k + 1} - x_k\right) + \frac{L\left(f_{h}\right)}{2}\left\|x_{k + 1} - x_k\right\|^2.
	\end{equation}
	Given the definition of maximum operator, we can combine \eqref{eq::max_varepsilon_L}-\eqref{eq::descent_lemma_applied_MO-IHT} to obtain 
	$f_{h}(x_{k + 1}) \le f_{h}(x_k) -\hat{\varepsilon} + \frac{L(f_{h}) - L}{2}\|x_{k + 1} - x_k\|^2$
	and, thus,
	\begin{equation}
		\label{eq::final_before_subseq}
		f_{h}(x_k) - f_{h}(x_{k + 1}) \ge \hat{\varepsilon} + \frac{L - L\left(f_{h}\right)}{2}\left\|x_{k + 1} - x_k\right\|^2.
	\end{equation}
	By Point 3. of Lemma \ref{lem::technical-lemma-MO-IHT}, we know that, for all $k \ge 1$, 
	$f_{h}(x_{k - 1}) \ge f_{h}(x_k)$. Given the latter result and \eqref{eq::final_before_subseq}, for $k \in K$ such that $k \ge 1$, we obtain
	\begin{equation}
		\label{eq::after_subseq}
		f_{h}(x_{k - 1}) - f_{h}(x_{k + 1}) \ge \hat{\varepsilon} + \frac{L - L\left(f_{h}\right)}{2}\left\|x_{k + 1} - x_k\right\|^2.
	\end{equation}
	Recursively applying the reasoning used for \eqref{eq::after_subseq}, we can conclude that for all $k_1, k_2 \in K$ such that $k_1 < k_2$, 
	\begin{equation}
		\label{eq::k1k2+1}
		f_{h}(x_{k_1}) - f_{h}(x_{k_2 + 1}) \ge \hat{\varepsilon} + \frac{L - L\left(f_{h}\right)}{2}\left\|x_{k_2 + 1} - x_{k_2}\right\|^2.
	\end{equation}
	Now, taking into account Point (v) of Lemma \ref{lem::technical-lemma-MO-IHT}, \eqref{eq::lim_barx} and the continuity of the norm, we get that $\lim_{{k \rightarrow \infty,k \in K}}x_{k + 1}  = \bar{x}.$
	Thus, taking the limit in \eqref{eq::k1k2+1} for $k \rightarrow \infty$, $k \in K$, recalling again Point (v) of Lemma \ref{lem::technical-lemma-MO-IHT}, \eqref{eq::lim_barx} and the continuity of $f_{h}$, we conclude that 
	$0 = f_{h}(\bar{x}) - f_{h}(\bar{x}) \ge \hat{\varepsilon}.$
	We get a contradiction since $\hat{\varepsilon} > 0$. Thus, the thesis is proved.
\end{proof}

\begin{remark}
	By Proposition \ref{prop:moiht-conv} and the continuity of $\theta_L$ (Lemma \ref{lemma:continuity-theta}) we are guaranteed that, for any $\varepsilon>0$, Algorithm \ref{alg::MO_IHT} will produce a point $x_k$ such that $\theta_L(x_k)>-\varepsilon$ in a finite number of iterations. Thus, we can effectively employ this condition as a practical stopping criterion for the \texttt{MOIHT} procedure.
\end{remark}

\subsection{Sparse Front Steepest Descent}
\label{subsec::SFSD}
In what follows, we describe and analyze the \textit{Sparse Front Steepest Descent} (\texttt{SFSD}) methodology. The algorithm can be seen as a two phases approach, which is based on the following consideration: in problems of the form \eqref{eq::mo-prob}, the Pareto front is usually an irregular set made up of several, distinct smooth parts; each of these nice portions of the front is typically the image of a set of solutions sharing the same structure, i.e., associated with the same support set. The rationale of the proposed algorithm is thus to first define a set of starting solutions; the support sets of these solutions should ideally be diverse and define a subspace where a portion of the Pareto set lies. Then, an adaptation of the front steepest algorithm \cite{Cocchi2020_Onthe,lapucciimproved} can be run starting from this initial set of solutions to span the front exhaustively. 
To the best of our knowledge, \texttt{SFSD} is the first front-oriented approach for cardinality-constrained MOO.

\subsubsection{Phase One: Initialization}
\label{subsubsec::initialization}

The first phase of the \texttt{SFSD} procedure deals with the identification of a set of starting solutions. 
The most direct way of proceeding would arguably be exhaustive enumeration of the super support sets, selecting for each a solution. However, the number of possible supports is high, growing as fast as $\binom{n}{s}$, but only a small fraction contributes to the Pareto front. Thus, this strategy is inefficient, up to being totally impractical with problems of nontrivial size. 

A totally random initialization might also appear to be a possible path to take, but, by similar reasons as above, it would end up being a completely luck-based operation. Therefore, we suggest to exploit single-point solvers to retrieve Pareto-stationary solutions. Indeed, by the mechanisms of this kind of algorithms, not only the obtained points are stationary but are usually also good solutions from a global optimization perspective. We explored the following (not exhaustive) list of options: 
\begin{itemize}
	\item Using the \texttt{MOIHT} discussed in Section \ref{subsec::MOIHT} in a multi-start fashion. Since the algorithm finds $L$-stationary solutions, optimization should be driven avoiding ``bad'' supports. 
	\item Using the Multi-Objective Sparse Penalty Decomposition (\texttt{MOSPD}) method from \cite{lapucci22} in a multi-start fashion; in brief, at each iteration $k$ of \texttt{MOSPD}, a pair $(x_{k + 1}, y_{k + 1})$ is found such that $x_{k + 1}$ is (approximately) Pareto-stationary for the penalty function $Q_{\tau_k}(x, y_{k + 1}) = F(x) + \frac{\tau_k}{2}\boldsymbol{1}_m\|x - y_{k + 1}\|^2,$
	with $\tau_k \rightarrow \infty$ for $k\to\infty$.
	The pair $(x_{k + 1}, y_{k + 1})$ is obtained by means of an \textit{alternate minimization} scheme. For further details, we refer the reader to \cite{lapucci22}. 
	\texttt{MOSPD} is proved to converge to points satisfying the multi-objective Lu-Zhang conditions for problem \eqref{eq::mo-prob}, that are even weaker than Pareto-stationarity; however, Penalty Decomposition methods have been shown to retrieve solutions both in the scalar \cite{kanzow23inexact} and in the multi-objective \cite{lapucci22} case that are good from a global optimization perspective.
	\item Combining the strategies of the two preceding points: for each point of a multi-start random initialization, we can first run the \texttt{MOSPD} procedure to exploit its exploration capabilities; then, we can use \texttt{MOIHT} in cascade, so that bad Lu-Zhang points that are not $L$-stationary can eventually be escaped. We refer to this approach as \texttt{MOHyb}.
	\item Solving the scalarized - single objective - problem for different trade-off parameters.
\end{itemize}

Once the starting set of solutions is obtained by one of the above strategies, a further step has to be carried out. Indeed, we need to associate each solution with a super support set. Now, if a solution has full support, then there is a unique super support and no ambiguity. However, there might be solutions with incomplete support; these solutions might be not Pareto-stationary (for example if obtained with \texttt{MOSPD}), in which case we shall carry out a descent step along the steepest feasible descent direction; if on the other hand we actually have a Pareto-stationary point with incomplete support, we shall associate it with any of the super supports.

Obviously, we can complete this first phase with a filtering operation, where dominated solutions get discarded.
To sum up, the result of the first phase of the algorithm provides a set of starting solutions each one associated with a super support set.

\subsubsection{Phase Two: Front steepest descent}
In Algorithm \ref{alg::Front}, we report the scheme of the proposed algorithmic framework (\texttt{SFSD}).

	\SetInd{1ex}{1ex}
\begin{algorithm}[h]
	\caption{\texttt{Sparse Front Steepest Descent}} \label{alg::Front}
	Input: $F:\mathbb{R}^n \rightarrow \mathbb{R}^m$, $\alpha_0 > 0$, $\delta \in (0, 1)$, $\gamma \in (0, 1)$.\\
	$\mathcal{X}^0$ = \texttt{Initialize}($F$)\\
	$k = 0$\\
	\While{a stopping criterion is not satisfied}{
		$\widehat{\mathcal{X}}^k = \mathcal{X}^k$ \\
		\ForAll{$(x_c, J_{x_c})\in \mathcal{X}^k$}{
			\If{$(x_c, J_{x_c}) \in \widehat{\mathcal{X}}^k$ }{
				\If{$\theta_{J_{x_c}}(x_c)<0$}{
					\vspace{0.05cm}$\alpha_c^k = \max\limits_{h\in\mathbb{N}} \{\alpha_0\delta^h\mid F(x_c+\alpha_0\delta^h d_{J_{x_c}}(x_c))\le F(x_c)+\boldsymbol{1}_m\gamma\alpha_0\delta^h\theta_{J_{x_c}}(x_c)\}$
				}
				\Else{$\alpha_c^k=0$ }
				$z_c^k = x_c+\alpha_c^kd_{J_{x_c}}(x_c)$\\
				$\widehat{\mathcal{X}}^k = \widehat{\mathcal{X}}^k \setminus \left\{(y, J_y) \in \widehat{\mathcal{X}}^k \mid J_y = J_{x_c},\; F(z^k_c) \lneqq F(y)\right\} \cup \left\{(z^k_c, J_{x_c})\right\}$\\
				\ForAll{$I\subseteq\{1,\ldots,m\}$ s.t.\ $\theta^I_{J_{x_c}}(z_c^k) < 0$}{
					\If{$(z_c^k, J_{x_c})\in\widehat{\mathcal{X}}^k$}{
						\vspace{0.05cm}$\alpha_c^I$ = $\max\limits_{h\in\mathbb{N}} \{\alpha_0\delta^h\mid \forall\,(y, J_y)\in\widehat{\mathcal{X}}^k,\; J_y=J_{x_c},\; \exists j \in \{1,\ldots, m\} \text{ s.t. } f_j(z_c^k+\alpha_0\delta^h d^I_{J_{x_c}}(z_c^k)) < f_j(y)\}$\\
						$\hat{z} = z_c^k + \alpha_c^I d^I_{J_{x_c}}(z_c^k)$\\
						$\widehat{\mathcal{X}}^k = \widehat{\mathcal{X}}^k \setminus \left\{(y, J_{y}) \in \widehat{\mathcal{X}}^k \mid J_y=J_{x_c},\; F(\hat{z}) \lneqq F(y)\right\} \cup \left\{(\hat{z}, J_{x_c})\right\}$\\
				}}
				
			}
			
		}
		$\mathcal{X}^{k + 1} = \widehat{\mathcal{X}}^k$ \\
		$k = k + 1$
	}
	\Return $\mathcal{X}^k$
\end{algorithm}

The method starts working with the starting set of solutions resulting from the \texttt{Initialize} step, i.e., phase one of the algorithm; the obtained set $\mathcal{X}^0$ is then given by 
\begin{equation*}
	\mathcal{X}^0 = \left\{\left(x, J_x\right) \mid J_x\in\mathcal{J}(x)\right\},
\end{equation*}
i.e., solutions associated with a corresponding super support set.
Given any pairs $(x,J_x), (y,J_y)\in\mathcal{X}^0$ with $J_x=J_y$, we assume that $x$ and $y$ are mutually nondominated w.r.t.\ $F$. 

Basically, the \texttt{SFSD} algorithm employs the instructions of the front steepest descent algorithm \cite{Cocchi2020_Onthe}, modified as suggested in \cite{lapucciimproved}, treating separately points associated with different super support sets.

Specifically, for any nondominated point $x_c$ in the current Pareto front approximation, a common descent step in the subspace corresponding to the support $J_{x_c}$ is carried out, doing a standard Armijo-type line search \cite{Fliege2000}.
In other words, the search direction is thus given by $d_{J_{x_c}}(x_c)$ according to \eqref{eq::lu_zhang}.
Then, further searches w.r.t. subsets of objectives are carried out from the obtained point, as long as it is not dominated by any other points $y$ in the set with $J_y=J_{x_c}$. These additional explorations are carried out along partial descent directions \cite{COCCHI2021100008,Cocchi2020_Onthe} in the reference subspace of the point at hand. Considering $I \subseteq \{1,\ldots, m\}$ as a subset of objectives indices, we define $\theta_J^I(\bar{x}) = \min_{d \in \mathbb{R}^n} \max_{j\in I}\, \nabla f_j(\bar{x})^\top d + \frac{1}{2}\|d\|^2\; \textit{s.t.}\; d_{\bar{J}} = \boldsymbol{0}_{|\bar{J}|}$.
Similar to \eqref{eq::lu_zhang}, the problem has a unique solution that we denote by $d_J^I(\bar{x})$.

Since the solutions are compared only if associated to the same super support set, the subspaces induced by different super support sets are explored separately in \texttt{SFSD}. As a result, we basically obtain separate Pareto front approximations, each one corresponding to a super support set. At the end of the \texttt{SFSD} execution, all the points can be compared and the dominated ones can finally be filtered out in order to obtain the final Pareto front approximation for problem \eqref{eq::mo-prob}.

Note that, conceptually, the algorithm can be seen as if multiple, independent runs of the front steepest descent method were carried out, each time constraining the optimization process in a particular subspace; however, exploration in \texttt{SFSD} is carried out ``in parallel'' throughout different supports, so that the front approximation is constructed uniformly and we can avoid cases where all the computational budget is spent for the optimization w.r.t.\ the first few considered supports.

\begin{remark}
	Since each point is considered for search steps only in the subspace induced by its associated super support set, it easily follows that every new solution will be feasible for \eqref{eq::mo-prob}. 
\end{remark}

\subsubsection{Algorithm Theoretical Analysis}
In this section, we state the convergence property of the \texttt{SFSD} methodology. We refer the reader to \cite{lapucciimproved} for the proofs of properties inherited by \texttt{SFSD} directly from the front steepest descent method.

Before proving the convergence result, we need to introduce the set $X_J^k = \left\{x \mid \exists \left(x, J\right) \in \mathcal{X}^k\right\},$ with $J$ denoting a super support set, and to recall the definition of \textit{linked sequences}, firstly introduced in \cite{liuzzi16}.

\begin{definition}
	Let $\left\{X_J^k\right\}$ be the sequence of sets of nondominated points, associated with the super support set $J$, produced by Algorithm \ref{alg::Front}. We define a linked sequence as a sequence $\left\{x_{j_k}^J\right\}$ such that, for all $k$, the point $x_{j_k}^J \in X_J^k$ is generated at iteration $k - 1$ of Algorithm \ref{alg::Front} by the point $x_{j_{k - 1}}^J \in X_J^{k - 1}$.
\end{definition}

\begin{proposition}
	Let us assume that $X_J^0$ is a set of mutually nondominated points and there exists $x_0^J \in X_J^0$ such that the set $\widehat{\mathcal{L}}_F\left(x_0^J\right) = \bigcup_{j=1}^m\{x \in \Omega \mid f_j(x) \le f_j(x_0^J)\}$ is compact. Moreover, let $\left\{X_J^k\right\}$ be the sequence of sets of nondominated points, associated with the super support set $J$, produced by Algorithm \ref{alg::Front}, and $\left\{x_{j_k}^J\right\}$ be a linked sequence. Then, the latter admits accumulation points, each one satisfying the MOLZ conditions for problem \eqref{eq::mo-prob}.
\end{proposition}
\begin{proof}
	By the instructions of the algorithm, each linked sequence $\left\{x_{j_k}^J\right\}$ can be seen as a linked sequence generated by applying the front steepest descent algorithm from \cite{lapucciimproved} to the problem of minimizing $F(x)$ subject to $x_{\bar{J}} = 0$. 
	Thus, we can follow the proof of \cite[Proposition 3.4]{lapucciimproved} to show that each accumulation point $\bar{x}$ of the linked sequence $\left\{x_{j_k}^J\right\}$ is such that $\theta_J\left(\bar{x}\right) = 0$, i.e., $\bar{x}$ satisfies the MOLZ conditions for \eqref{eq::mo-prob}.
\end{proof}

\section{Computational Experiments}
\label{sec::prel_ext}
In this section, we report the results of some computational experiments aimed at assessing the numerical potential of the proposed approaches. The code for the experiments, which was written in \texttt{Python3}, was run on a computer with the following characteristics: Ubuntu 22.04, Intel
Xeon Processor E5-2430 v2 6 cores 2.50 GHz, 16 GB RAM. In order to solve instances of problems like \eqref{eq::basic_feasibility_MOO}-\eqref{eq::lu_zhang}-\eqref{eq::MO-IHT-problem}, the Gurobi optimizer (version 9.1) was employed.

\subsection{Experiments Setup}

In our numerical experience, we considered two classes of problems:
cardinality-constrained quadratic problems and sparse logistic regression tasks.

The quadratic MOO problems, which often represent a useful test benchmark in optimization, have the form
\begin{equation*}
	\min_{x\in \mathbb{R}^n}\quad \tfrac{1}{2}\left(x^\top Q_1 x - c_1^\top x, x^\top Q_2 x - c_2^\top x\right)^\top\quad\text{s.t.}\quad\|x\|_0 \le s,
\end{equation*}
where $Q_1, Q_2 \in \mathbb{R}^{n \times n}$ are random positive semi-definite matrices and $c_1, c_2 \in \mathbb{R}^n$ are vectors whose values are randomly sampled in the range $[-1, 1)$. In the experiments, we varied the following problem parameters: the size $n \in \{10, 25, 50\}$; the condition number of the matrices $\kappa \in \{1, 10, 100\}$; the cardinality upper bound $s$. In particular, the latter was set in the following way: for $n = 10$, $s \in \{2, 5, 8\}$; for $n = 25$, $s \in \{5, 10, 20\}$; for $n = 50$, $s \in \{5, 15, 30\}$. Moreover, we used 3 different seeds for the pseudo-random number generator, thus leading to a total of 81 quadratic problems. For each instance, $Q_1$ and $Q_2$ are characterized by the same condition number, i.e., $L(f_1)=L(f_2)=\kappa$.

As for the sparse logistic regression problem \cite{bertsimas17,Civitelli2021}, it is a relevant task in machine and statistical learning. Given a dataset of $N$ samples with $n$ features $R = (r_1,\ldots, r_N)^\top \in \mathbb{R}^{N \times n}$ and $N$ corresponding labels $\{t_1,\ldots, t_N\}$ belonging to $\{-1, 1\}$, the \textit{regularized sparse logistic regression problem} is given by:
\begin{equation*}
	\min_{w\in\mathbb{R}^n}\;\tfrac{1}{N}\sum_{i=1}^{N}\log\left(1 + \exp\left(-t_i\left(w^\top r_i\right)\right)\right) + \tfrac{\lambda}{2} \left\|w\right\|^2\quad \text{ s.t. } \quad \|w\|_0\le s,
\end{equation*}
where $\lambda \ge 0$. The logistic loss aims to fit the training data, while the regularization term helps to avoid overfitting. The two functions are clearly in contrast with each other. For our experiments, we employed the multi-objective reformulation considered in \cite{lapucci22}:
\begin{equation*}
	\min_{w\in\mathbb{R}^n}\quad \bigg(\tfrac{1}{N}\sum_{i=1}^{N}\log(1 + \exp(-t_i(w^\top r_i))),\; \tfrac{1}{2}\|w\|^2\bigg)^\top \quad \text{ s.t. } \quad \|w\|_0\le s.
\end{equation*}
For this problem, $L(F) = (\|R^\top R\|\mathbin{/}N, 1)^\top$.
The dataset suite we considered is composed of 7 binary classification datasets from the UCI Machine Learning Repository \cite{Dua2019} (Table \ref{tab::datasets}). We tested the algorithms on instances of the problem with $s \in \{2, 5, 8, 12, 20\}$. For each dataset, the samples with missing values were removed. Moreover, the categorical variables were one-hot encoded, while the other ones were standardized to zero mean and unit standard deviation.

In order to evaluate the performance of the algorithms one compared to the others, we employed the performance profiles \cite{dolan2002benchmarking}. In brief, performance profiles show the probability that a metric value achieved by a method in a problem is within a factor $\tau \in \mathbb{R}$ of the best value obtained by any of the solvers considered in the comparison. We refer the reader to \cite{dolan2002benchmarking} for more details. As performance metrics, we used some classical ones of the multi-objective literature: \textit{purity}, $\Gamma$\textit{--spread}, $\Delta$\textit{--spread} \cite{custodio11} and \textit{hyper-volume} \cite{zitzler98}. Note that, since \textit{purity} and \textit{hyper-volume} have increasing values for better solutions, the performance profiles w.r.t. them were produced based on the inverse of the obtained values.

Note that, for both classes of problems, in the following we will also consider solution approaches based on scalarization, i.e., tackling the problem $\min_{x \in \Omega}\;f_1(x) + \lambda f_2(x),$
where $\lambda \ge 0$. In the quadratic case, the problem can be solved by means of commercial solvers such as Gurobi, exploiting an \texttt{MIQP} reformulation. In the logistic regression case, we instead use the \textit{Greedy Sparse-Simplex} (\texttt{GSS}) algorithm \cite{beck13}.
Note that, opposed to \texttt{MIQP} approach in quadratic problems, \texttt{GSS} is not guaranteed to produce a Pareto optimal solution.

As anticipated in Section \ref{subsubsec::initialization}, our \texttt{SFSD} methodology was tested taking as starting solutions the ones generated by the single-point methods mentioned above, i.e., \texttt{MOIHT}, \texttt{MOSPD}, \texttt{MOHyb}, \texttt{MIQP} and \texttt{GSS}. Similarly to what is done in \cite{lapucciimproved}, we employed a strategy to limit the number of points used for partial descent searches, in order to improve the \texttt{SFSD} efficiency and prevent the production of too many, very close solutions. In detail, we added a condition based on the crowding distance \cite{deb02} to determine whether a point should be considered for further exploration after the common descent step or not.

Every execution had a time limit of 4 minutes. In particular, each single-point method was tested in a multi-start fashion: it had to process as many input points as possible within 2 minutes; in the remaining time, the \texttt{MOSD} procedure was employed as a refiner, starting at each returned point and keeping fixed its zero variables so that the cardinality constraint was kept valid. In \texttt{SFSD}, we set a time limit of 2 minutes for both phases of the algorithm.
For \texttt{MOIHT}, \texttt{MOSPD} and \texttt{MOHyb}, we considered $2n$ initial solutions randomly sampled from a box ($[-2, 2]^n$ for the quadratic problems; $[0, 1]^n$ for logistic regression). In order to be feasible, each initial point is first projected onto $\Omega$. These algorithms were executed 5 times with different seeds for the pseudo-random number generator to reduce the sensibility from the random initialization. The five generated fronts were then compared based on the \textit{purity} metric and only the best and worst ones were chosen for the comparisons. The scalarization-based approaches were run once considering $2n$ values for $\lambda$, i.e., $\lambda \in \{2^{i + \frac{1}{2}} \mid i \in \mathbb{Z},\; i \in [-n, n)\}$, and starting at the initial solution $\boldsymbol{0}_n \in \Omega$.

\begin{table}[htb]
	\footnotesize
	\centering
	\begin{tabular}{|c||c|c|}%
		\hline%
		\textbf{DATASET}&$N$&$n$\\%
		\hline%
		\hline%
		Heart (Statlog)&270&25\\%
		\hline%
		Breast Cancer Wisconsin (Prognostic)&194&33\\%
		\hline%
		QSAR Biodegradation&1055&41\\%
		\hline
		SPECTF heart&267&44\\%
		\hline
		Spambase&4601&57\\%
		\hline
		Optical recognition of handwritten digits&3823&62\\%
		\hline
		Madelon&2000&500\\%
		\hline%
	\end{tabular}
	\vspace{0.2cm}
	\caption{Datasets used for the experiments on sparse logistic regression. The number of features takes into account one-hot encoding of categorical ones.}
	\label{tab::datasets}
\end{table}

\subsection{Quadratic Problems}

In this section, we report the results on the cardinality-constrained quadratic problems.
As for the algorithms parameters, based on some preliminary experiments not reported here for the sake of brevity, we set: $\varepsilon = 10^{-7}$, $L = 1.1\kappa$ for \texttt{MOIHT}; $\tau_{k + 1} = 1.5\tau_k$, $\varepsilon_0 = 10^{-2}$, $\varepsilon_{k + 1} = 0.9\varepsilon_k$ and $\| x_{k + 1} - y_{k + 1}\| \le 10^{-3}$ as stopping condition for \texttt{MOSPD}; the Pareto stationarity approximation degree $\varepsilon = 10^{-7}$ for \texttt{MOSD}; $\alpha_0 = 1$, $\delta = 0.5$ and $\gamma = 10^{-4}$ for all Armijo-type line searches.
Possible values for the \texttt{MOSPD} parameter $\tau_0$ are discussed in the next section. The parameters choices for \texttt{MOIHT} and \texttt{MOSPD} were also used in \texttt{MOHyb}.

\subsubsection{Preliminary Assessment of \texttt{MOIHT}, \texttt{MOSPD} and \texttt{MOHyb}}
\label{subsubsec:prel_eval}

We start analyzing the effectiveness of \texttt{MOIHT}, \texttt{MOSPD} and \texttt{MOHyb}, comparing them in Figure \ref{fig::QP_2} on a selection of quadratic problems. 
In order to show the differences among the algorithms as clearly as possible, only for this experiment, we considered a single run where the methods took as input the same 25 randomly extracted initial points. Moreover, we set no time limit, so that all the algorithms could process each initial solution until the respective stopping criteria were met. 

The \texttt{MOSPD} and \texttt{MOHyb} performance was investigated for values for $\tau_0 \in \{1, 100\}$ (results for $\tau_0 = 100$ are shown in the left column of the figure, $\tau_0 = 1$ on the right). The black dots indicate the reference front: the latter is obtained combining the fronts retrieved by running \texttt{SFSD} with all the proposed initialization strategies and discarding the dominated solutions.

\begin{figure*}
	\centering
	\subfloat[]{\includegraphics[width=0.375\textwidth]{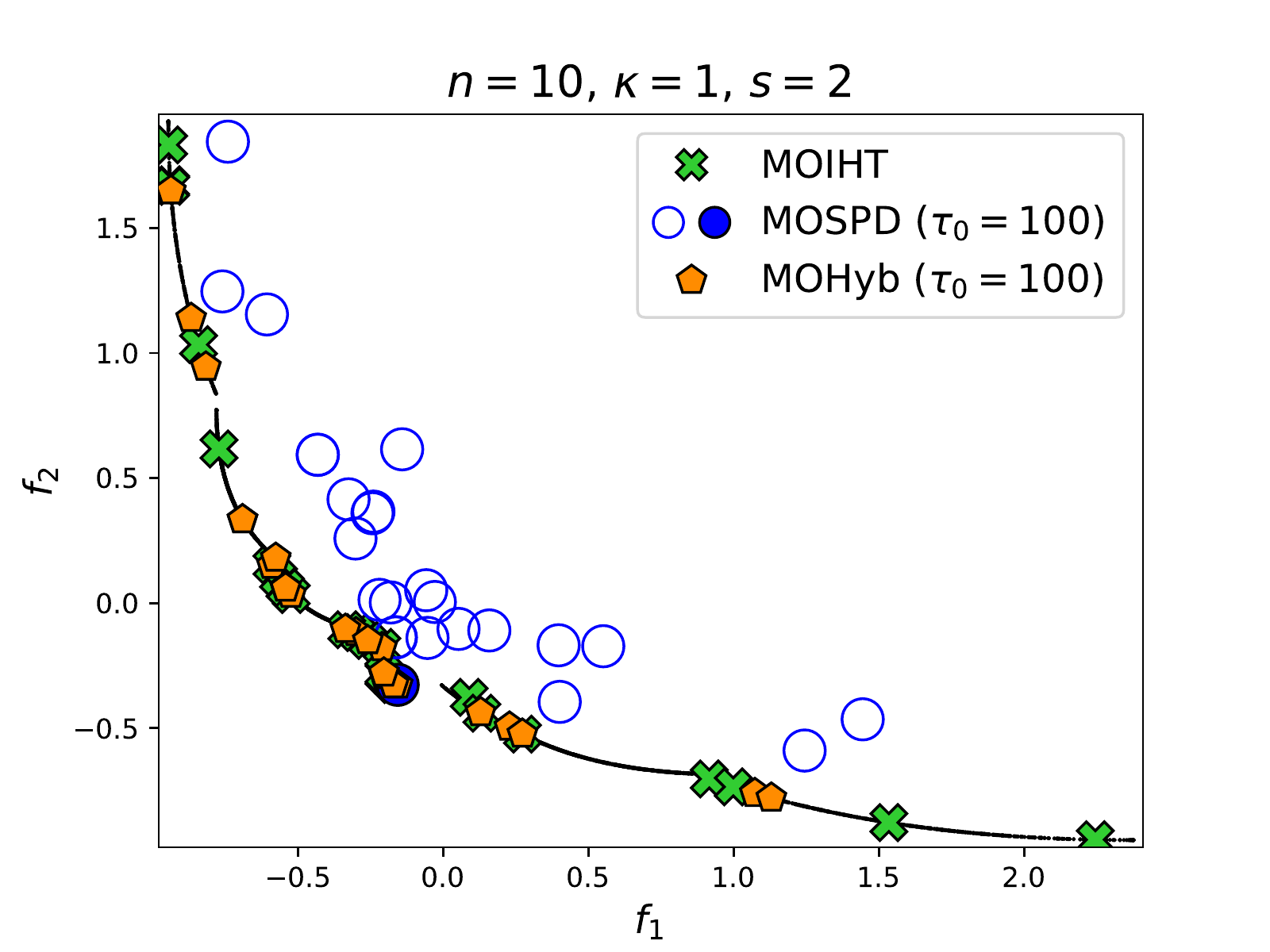}}
	\hfil
	\subfloat[]{\includegraphics[width=0.375\textwidth]{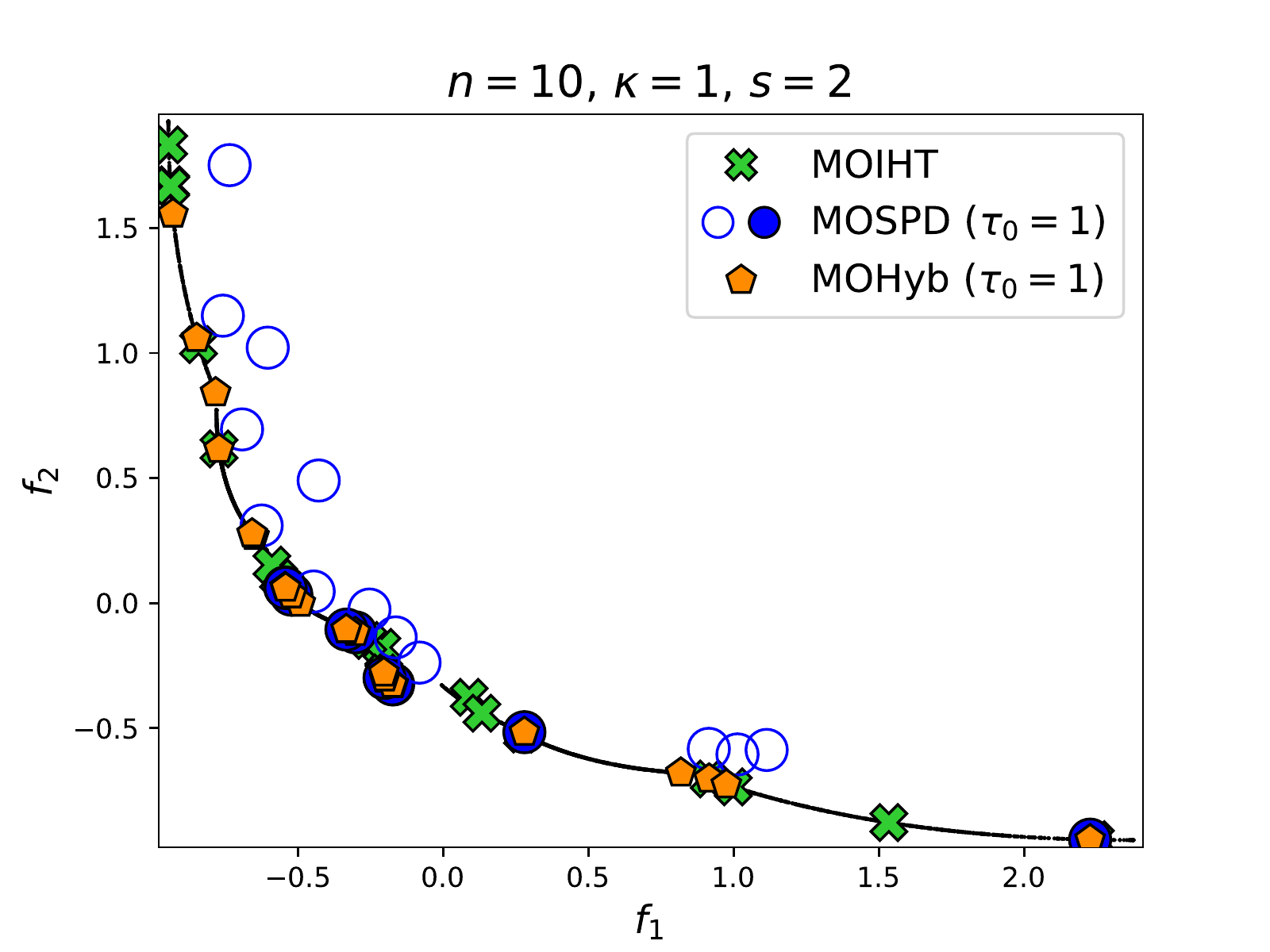}}
	\\
	\subfloat[]{\includegraphics[width=0.375\textwidth]{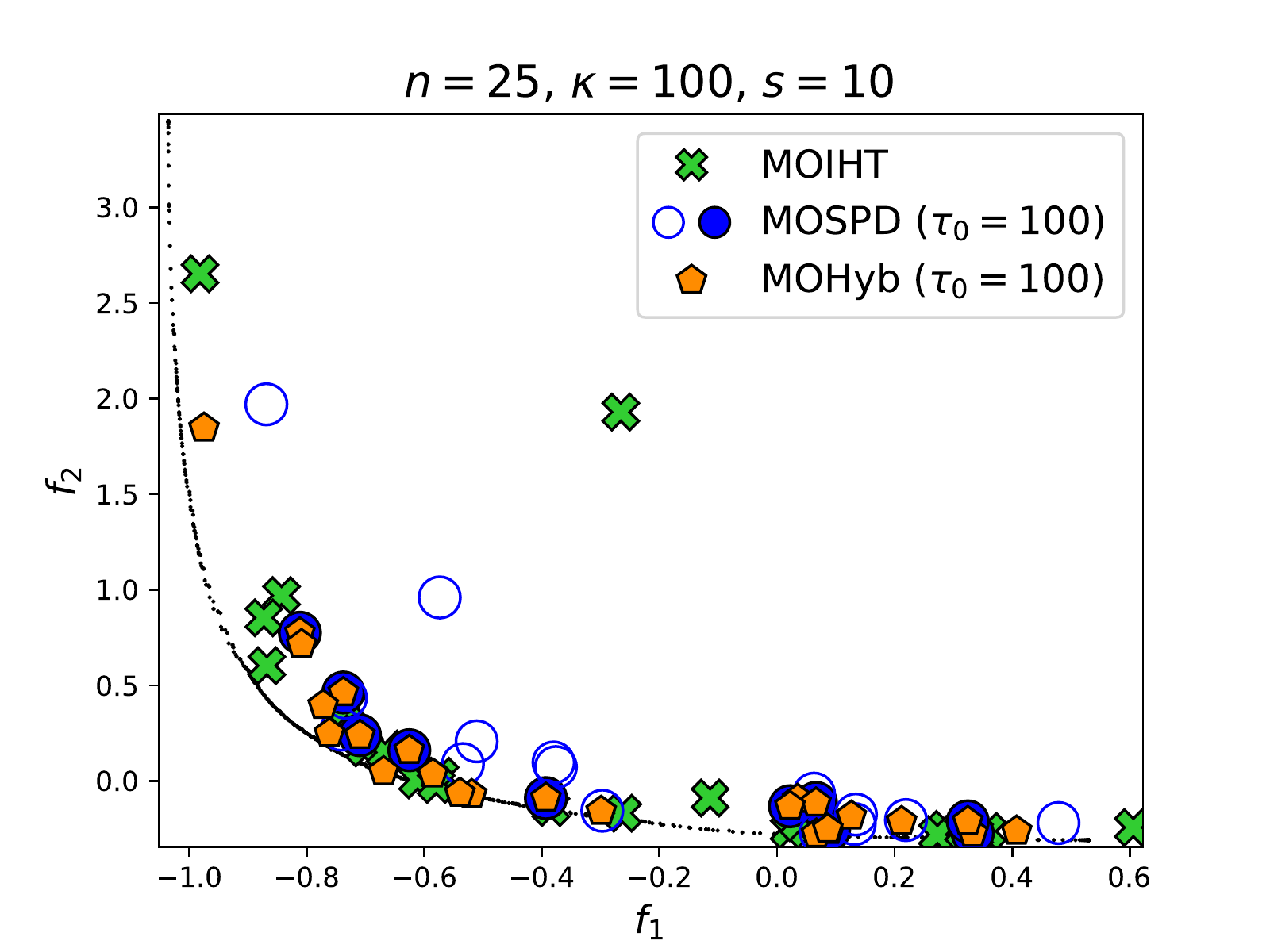}}
	\hfil
	\subfloat[]{\includegraphics[width=0.375\textwidth]{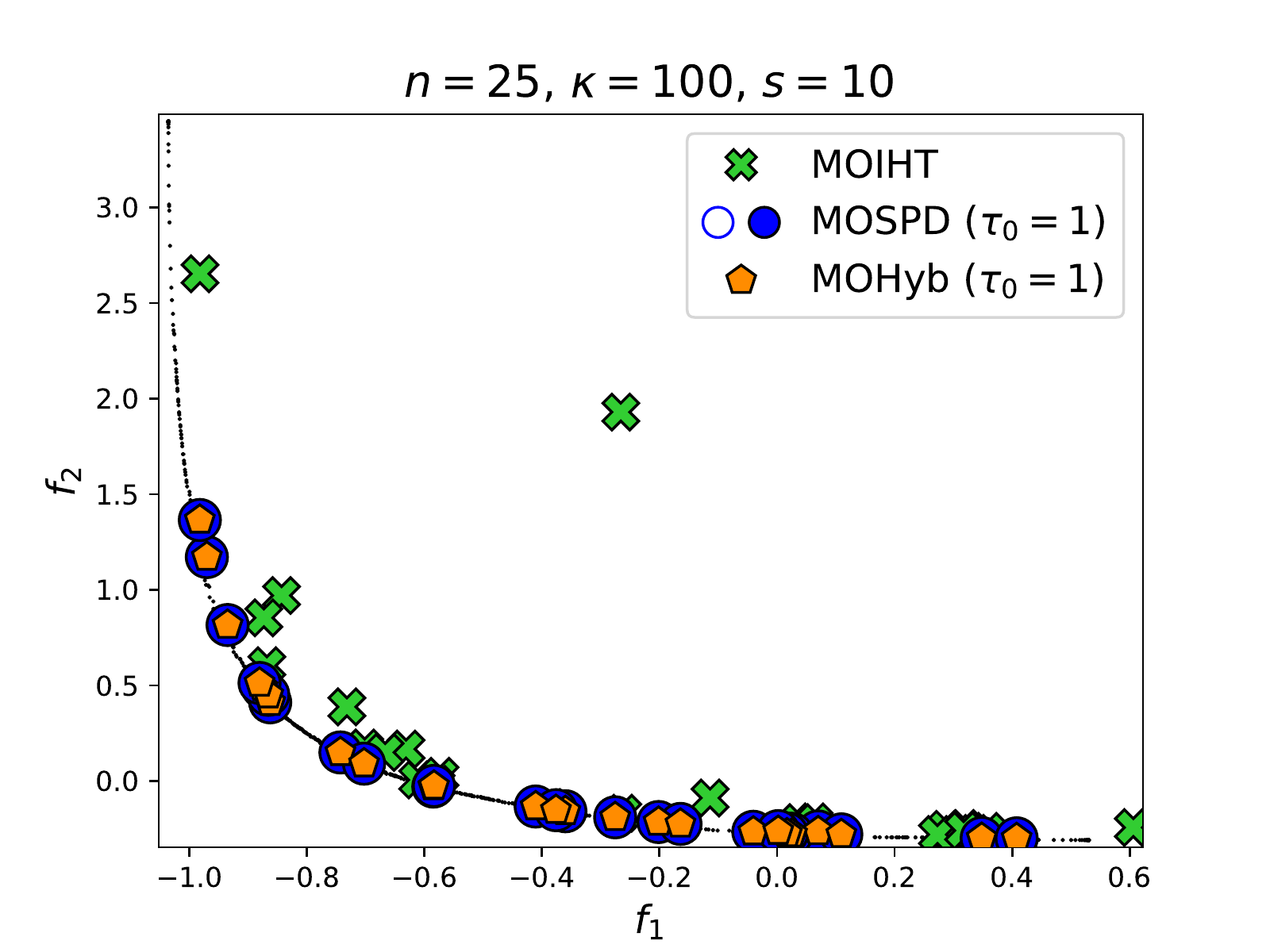}}
	\\
	\subfloat[]{\includegraphics[width=0.375\textwidth]{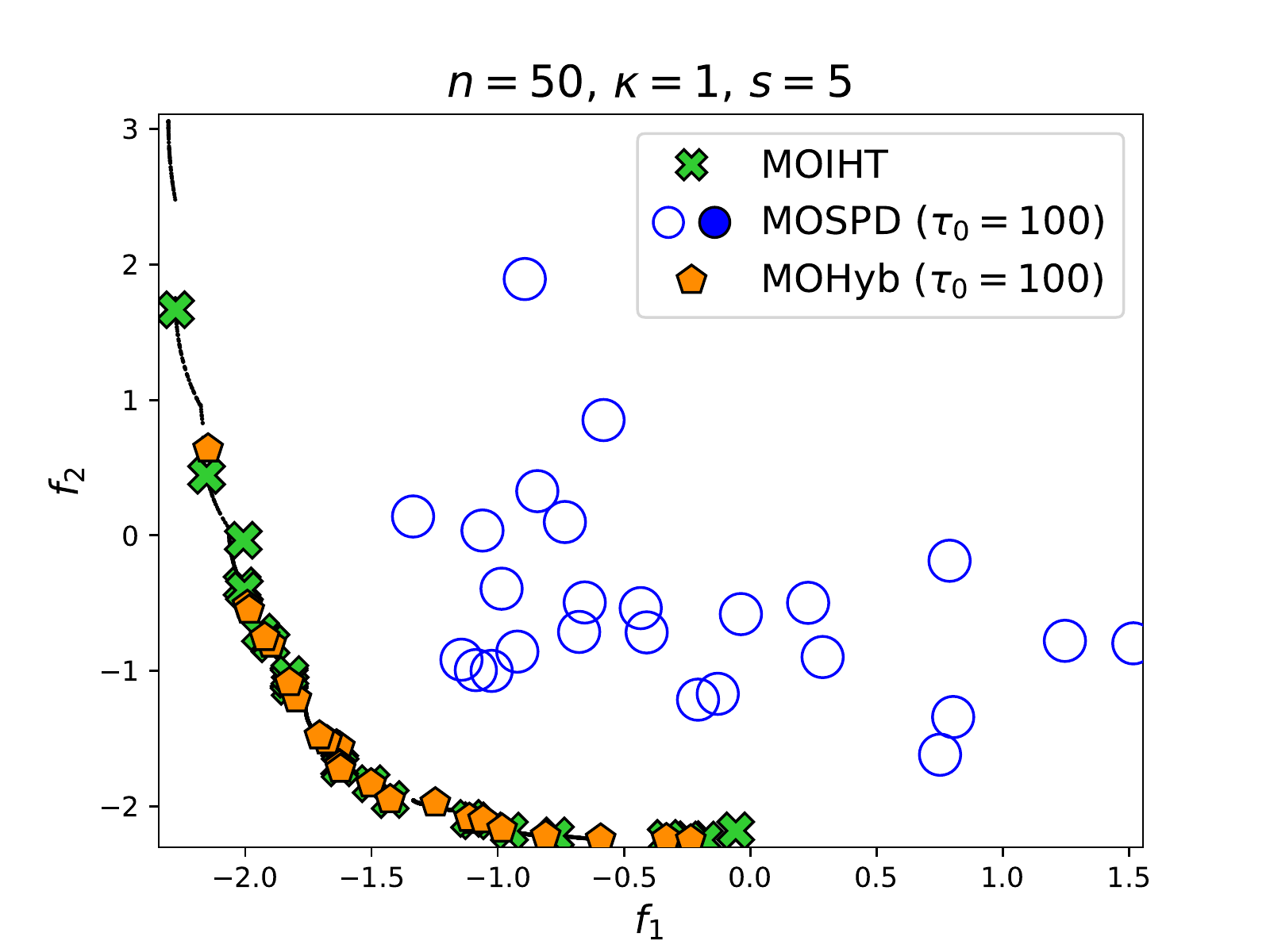}}
	\hfil
	\subfloat[]{\includegraphics[width=0.375\textwidth]{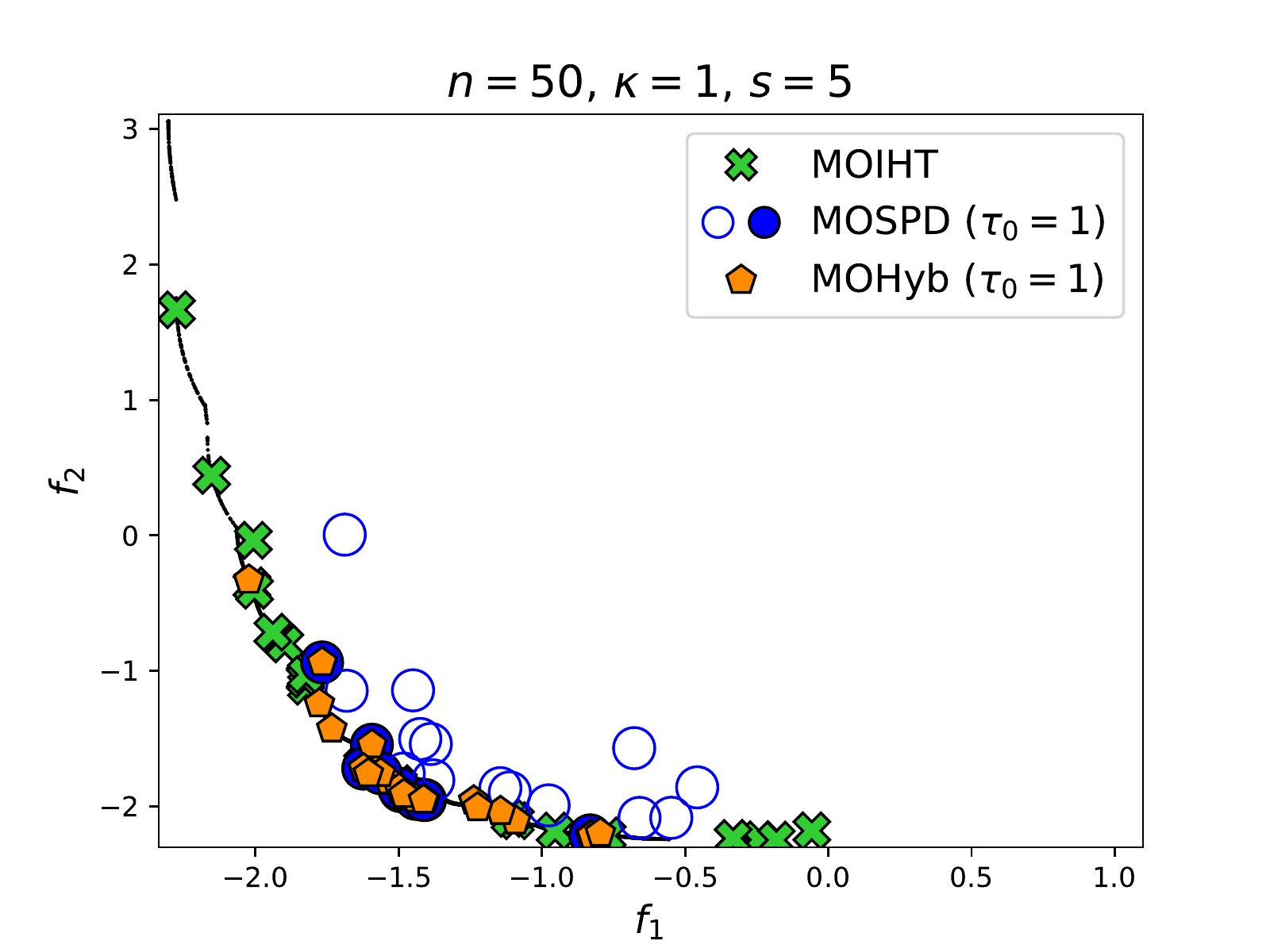}}
	\\
	\subfloat[]{\includegraphics[width=0.375\textwidth]{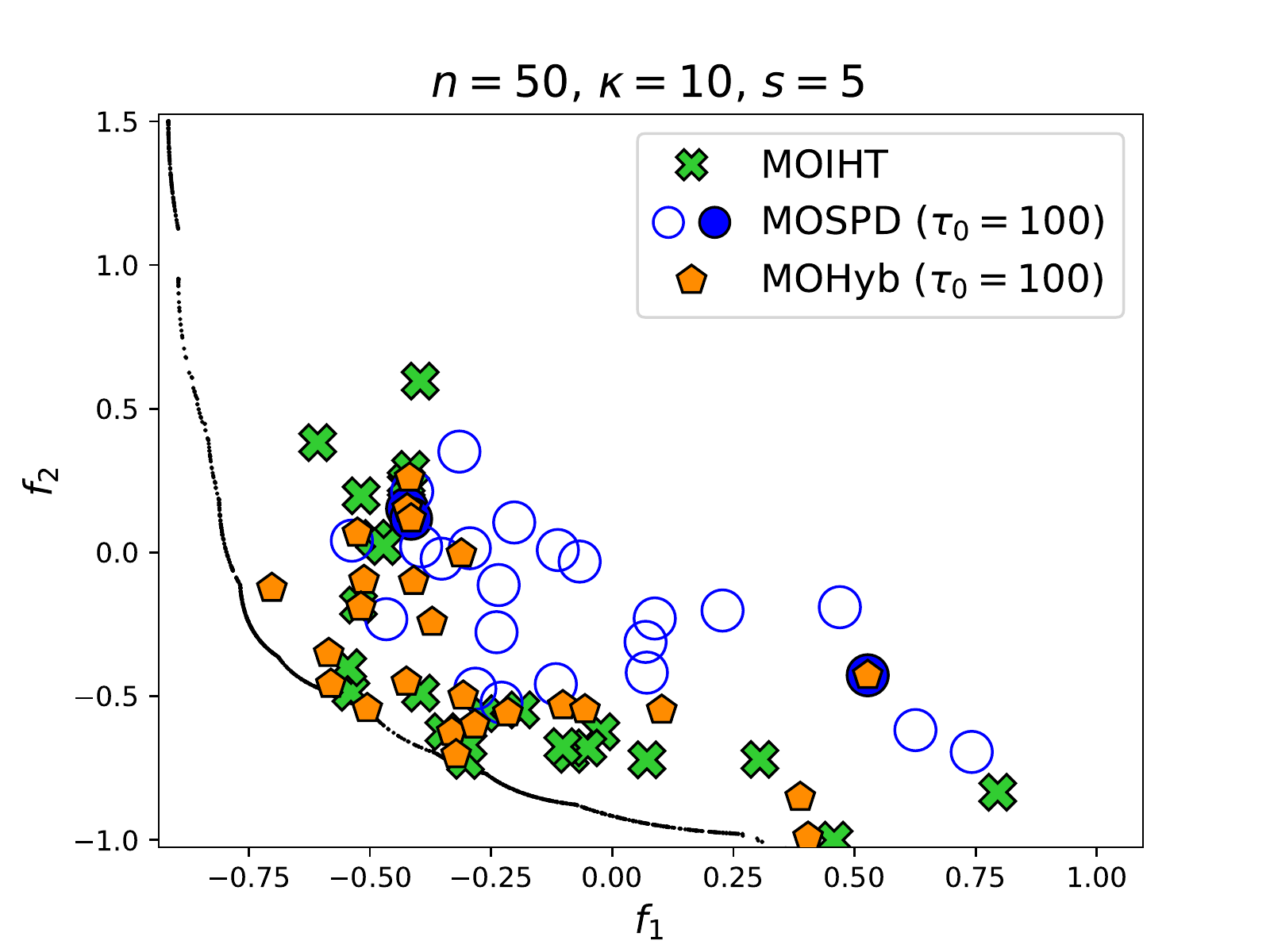}}
	\hfil
	\subfloat[]{\includegraphics[width=0.375\textwidth]{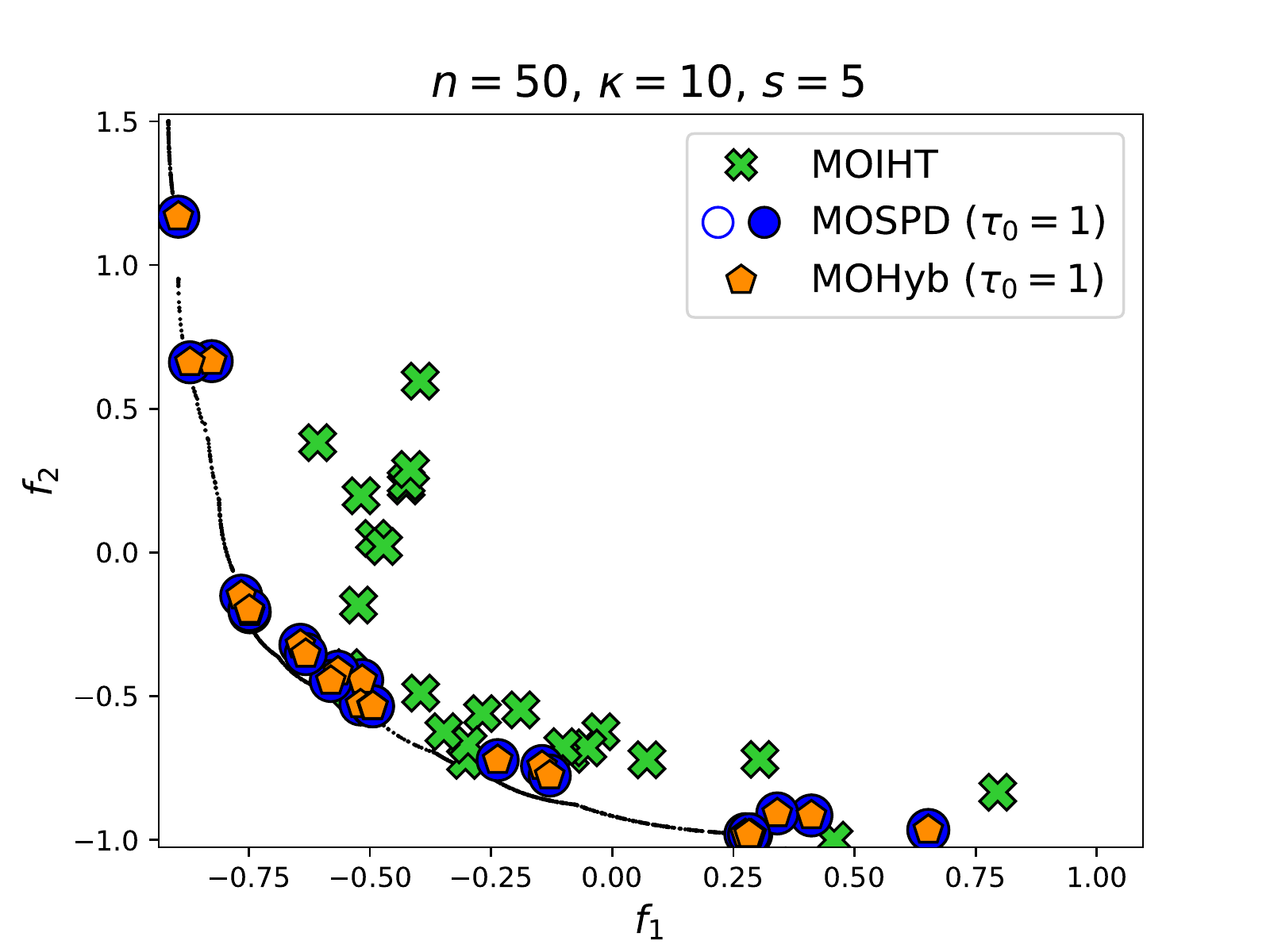}}
	\caption{Results achieved by \texttt{MOIHT}, \texttt{MOSPD} and \texttt{MOHyb} with $\tau_0 \in \{1, 100\}$, starting at 25 random initial solutions, on a selection of quadratic problems. The filled markers denote $L$-stationary solutions ($L = 1.1\kappa$). The small black dots form the reference front.}
	\label{fig::QP_2}
\end{figure*}

In well-conditioned problems ($\kappa = 1$), the \texttt{MOIHT} algorithm performs well, reaching solutions that belong to the reference front. As for \texttt{MOSPD}, the results with $\tau_0 = 100$ are worse, with \texttt{MOSPD} obtaining solutions far from the reference front. The situation is further stressed as the problem dimension $n$ grows. This sounds reasonable, since setting $\tau_0 = 100$ in Penalty Decomposition schemes binds the variables close to the initial feasible solution and, as a consequence, prevents from exploiting the exploration capabilities of \texttt{MOSPD}. A better choice for $\tau_0$ ($\tau_0 = 1$) improves the performance of the algorithm, although \texttt{MOIHT} still performs better. This result is somewhat in line with the theory: \texttt{MOIHT} generates $L$-stationary points, whereas \texttt{MOSPD} converges to solutions only guaranteed to satisfy the (weaker) MOLZ conditions. 
In this scenario, \texttt{MOHyb} inherits the effectiveness of \texttt{MOIHT}: regardless the value for $\tau_0$, it succeeded in getting solutions of the reference front.

In ill-conditioned problems ($\kappa > 1$), the \texttt{MOIHT} performance gets worse: the method struggled in reaching the reference front. This might be explained by the larger values of $L$ that have to be used with these problems ($L = 1.1\kappa$). As the value of $L$ grows, the $L$-stationarity condition does not provide enough information on the quality of the solution support set. As a consequence, \texttt{MOIHT} can end up in many $L$-stationary points with ``bad'' support, i.e., far from the actual Pareto front of the problem. \texttt{MOSPD} with $\tau_0 = 1$ obtained better solutions in these cases. Employing lower values for $\tau_0$, the approach is initially allowed to search for a good point minimizing $F$: this feature can be crucial to avoid a large portion of ``bad'' $L$-stationary points and to reach solutions in the reference front. 
\texttt{MOHyb} ($\tau_0 = 1$) proved to be effective in these scenarios too. The hybrid approach, in these ill-conditioned cases, profited from the exploration capabilities of \texttt{MOSPD}, reaching the same solutions. However, like in the well-conditioned case, \texttt{MOHyb} also proved to be less sensitive than \texttt{MOSPD} w.r.t.\ the value of $\tau_0$, taking advantage of the \texttt{MOIHT} mechanisms to reach at least $L$-stationary solutions when $\tau_0=100$.

\subsubsection{Evaluation of the \texttt{SFSD} Methodology}

We start the analysis of the \texttt{SFSD} algorithm performance on the quadratic problems through Figure \ref{fig::Front}, where we show how the front descent phase allows to improve the results of basic multi-start approaches corresponding to phase one, i.e., \texttt{MOIHT}, \texttt{MOSPD}, \texttt{MOHyb} and \texttt{MIQP}. According to the results of Section \ref{subsubsec:prel_eval}, for \texttt{MOSPD} and \texttt{MOHyb}, we set $\tau_0 = 1$.

\begin{figure*}
	\centering
	\subfloat[\texttt{MOIHT}]{\includegraphics[width=0.375\textwidth]{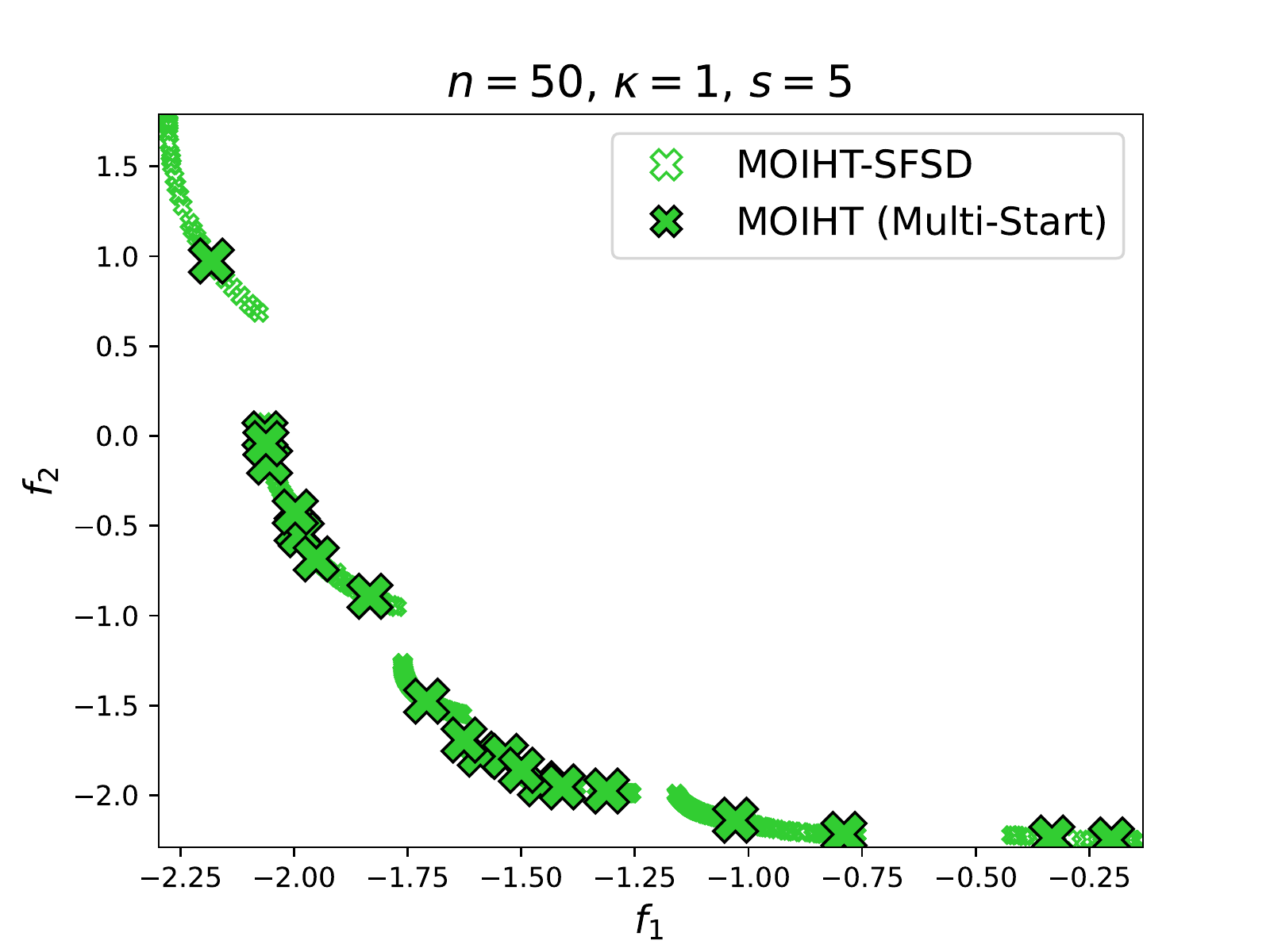}}
	\hfil
	\subfloat[\texttt{MOHyb}]{\includegraphics[width=0.375\textwidth]{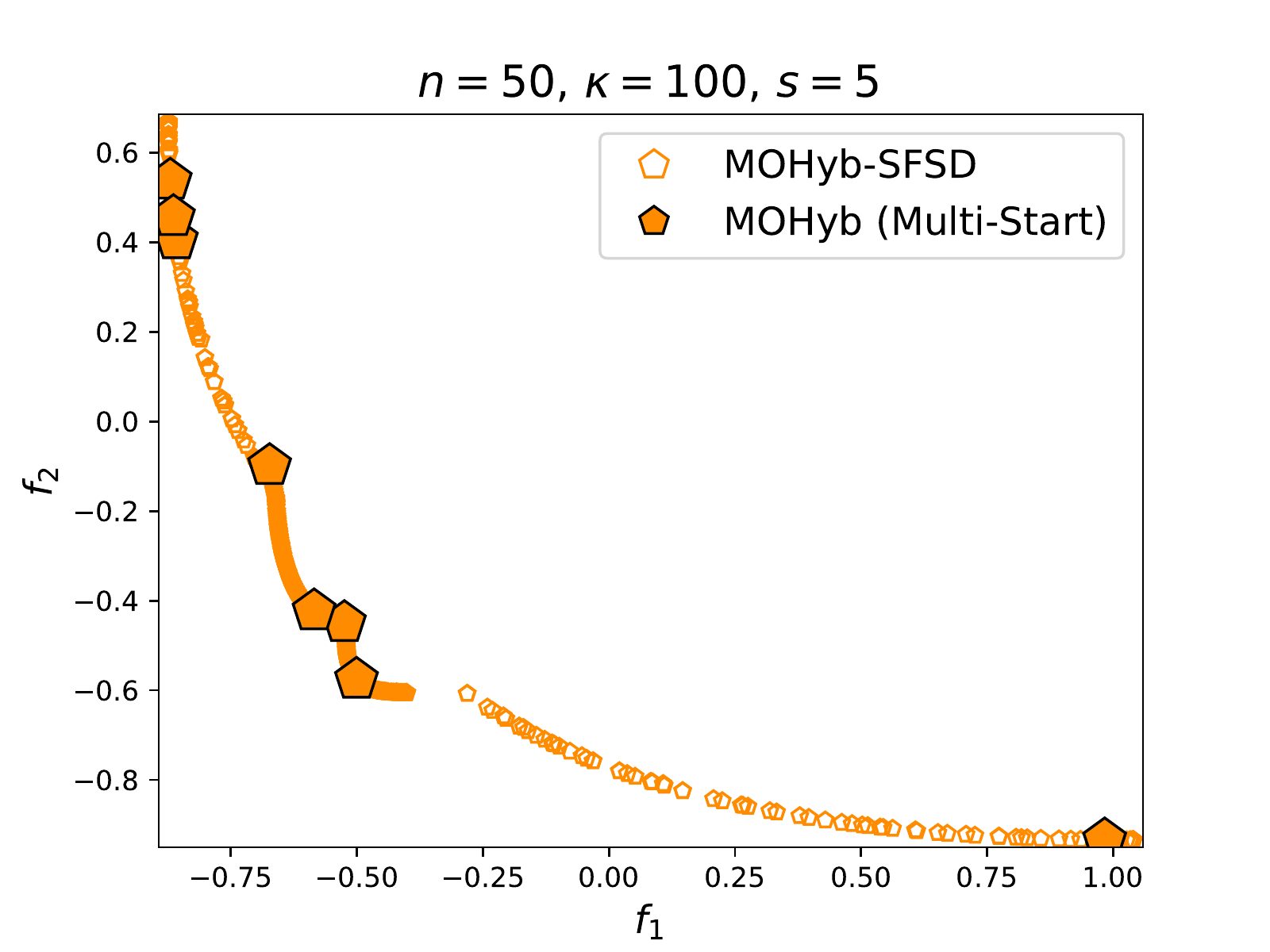}}
	\\
	\subfloat[\texttt{MOSPD}]{\includegraphics[width=0.375\textwidth]{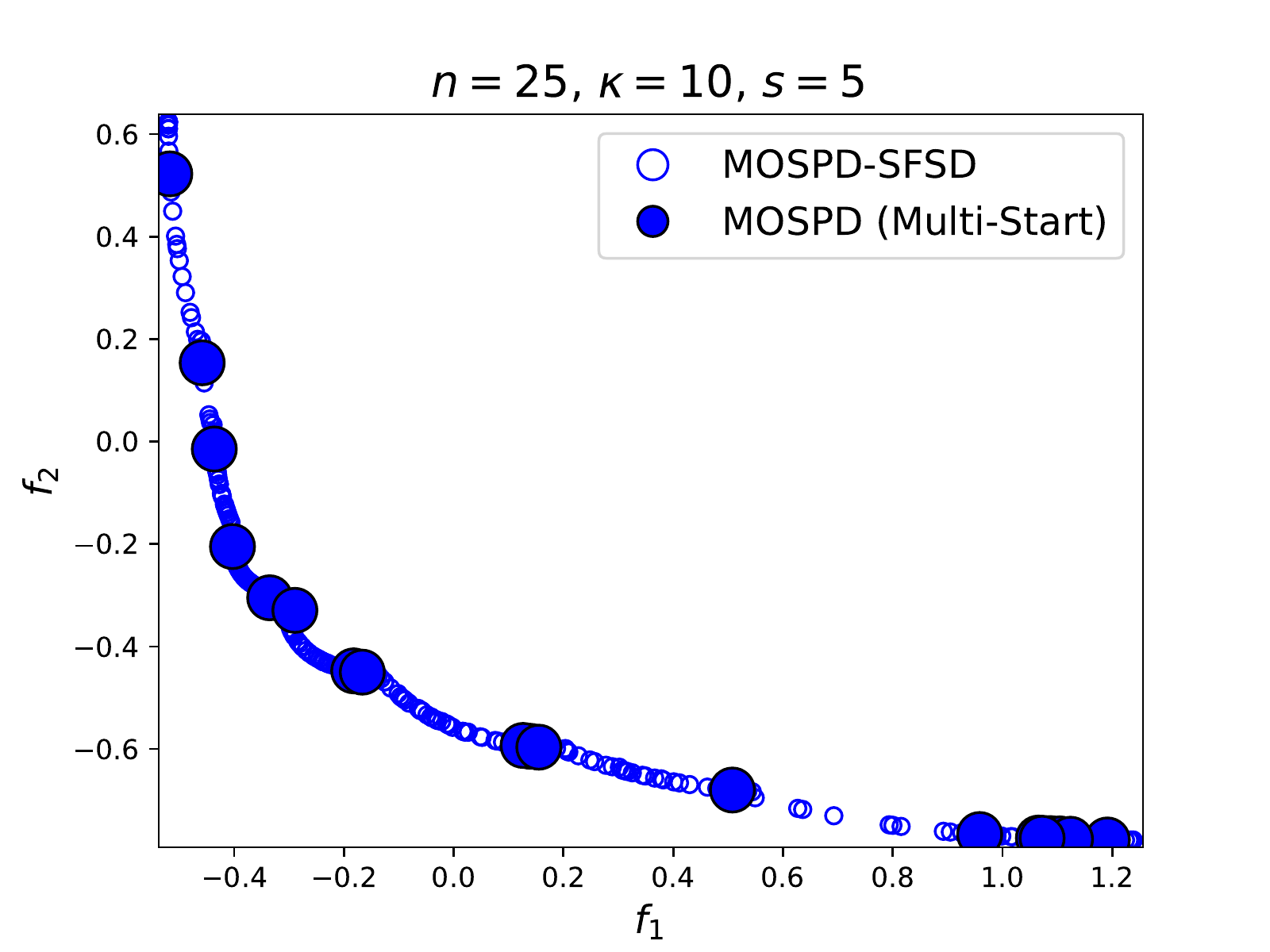}}
	\hfil
	\subfloat[\texttt{MIQP}]{\includegraphics[width=0.375\textwidth]{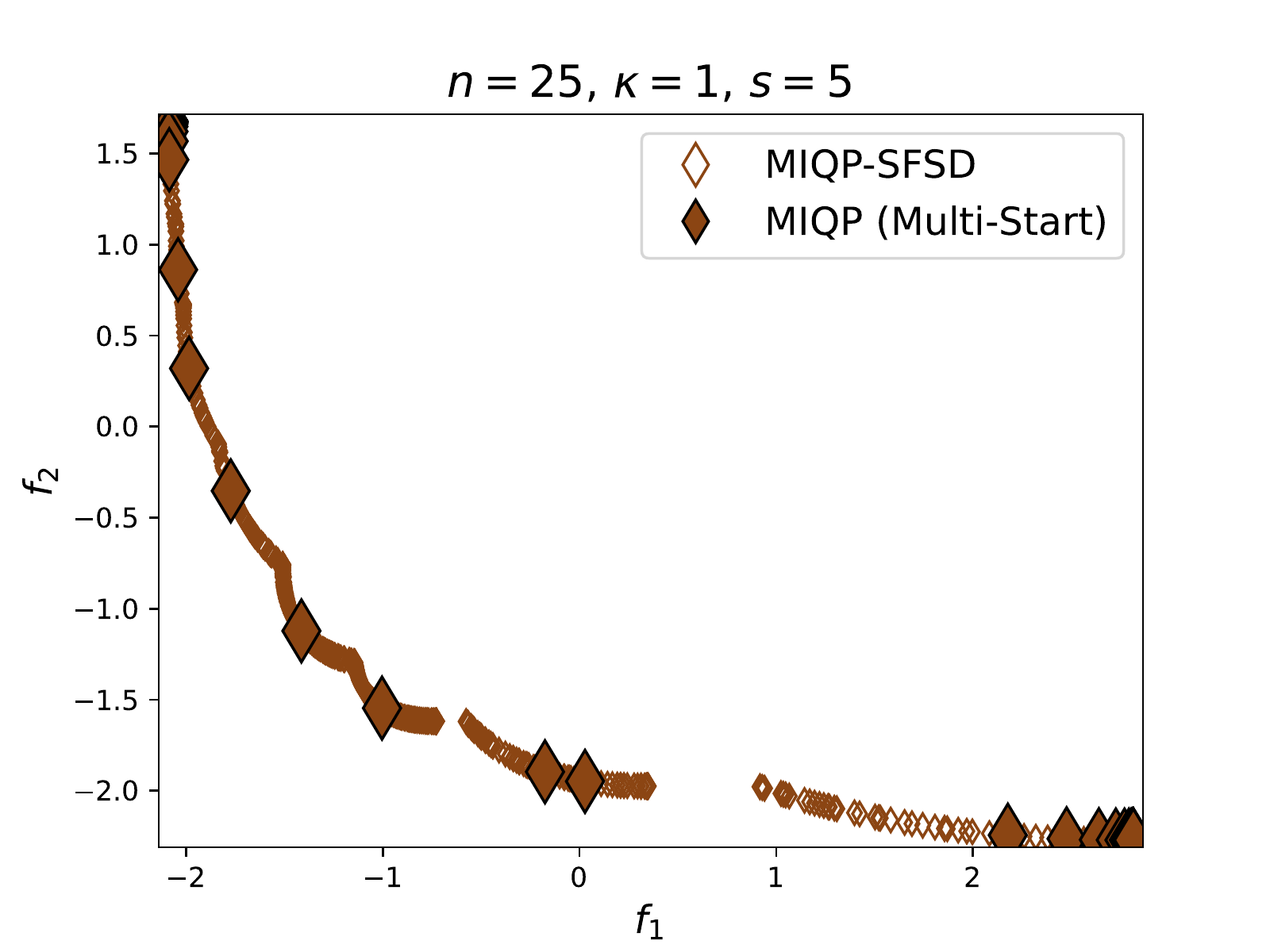}}
	\caption{Results of \texttt{SFSD} phase two compared to simple \texttt{MOSD} refinement of solutions retrieved in phase one. We show one example instance for each considered multi-start/phase one strategy.}
	\label{fig::Front}
\end{figure*}

As anticipated in Section \ref{subsec::SFSD}, in cardinality-constrained MOO the Pareto fronts are typically irregular and made up of several smooth parts. The plots in the figure perfectly reflect these characteristics: each front portion can indeed be associated with a specific support set. Starting from the solutions generated by the single-point methods, the \texttt{SFSD} methodology proves to be effective in exhaustively spanning each portion associated with a support set. This feature allowed our novel front-oriented approach to identify regions of the Pareto front that would have otherwise be hardly covered with the multi-start strategy.

As mentioned in Section \ref{subsec::SFSD}, to the best of our knowledge, \texttt{SFSD} is the first front-oriented approach for cardinality-constrained MOO. In the absence of other specialized algorithms, it is difficult to quantitatively assess the potential of our algorithm and we need to resort to the visual inspection of the solutions. 
In the rest of the section, we then focus our attention on the different options we outlined for the phase one of the \texttt{SFSD} algorithm, comparing its performance as the initialization strategy varies. The comparisons were made by means of the performance profiles (Figure \ref{fig::QP_PP}) on the entire benchmark of quadratic problems.
 
\begin{figure*}
	\centering
	\subfloat[\label{fig::QP_PP_Purity}]{\includegraphics[width=0.25\textwidth]{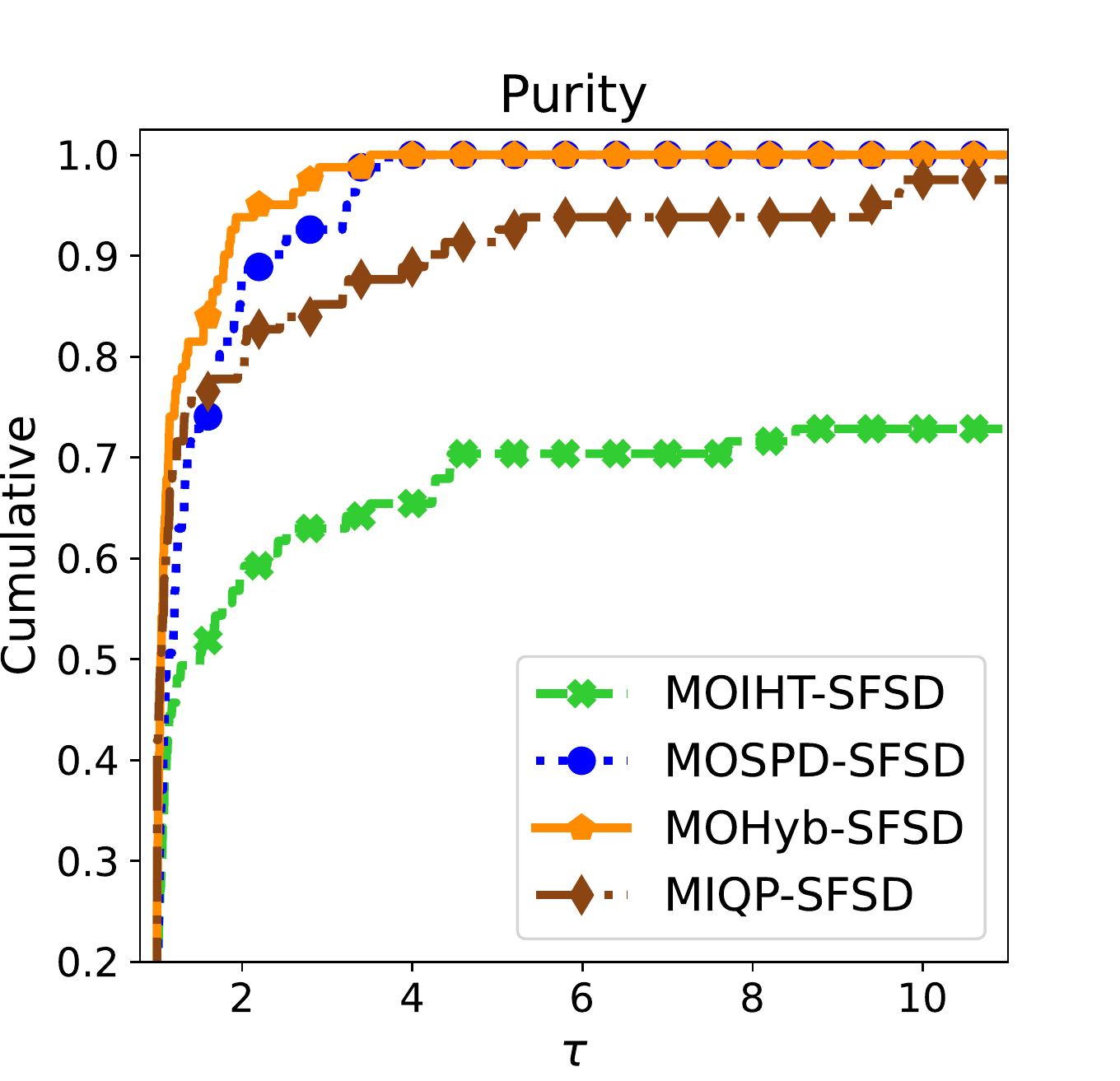}}
	\hfil
	\subfloat[]{\includegraphics[width=0.25\textwidth]{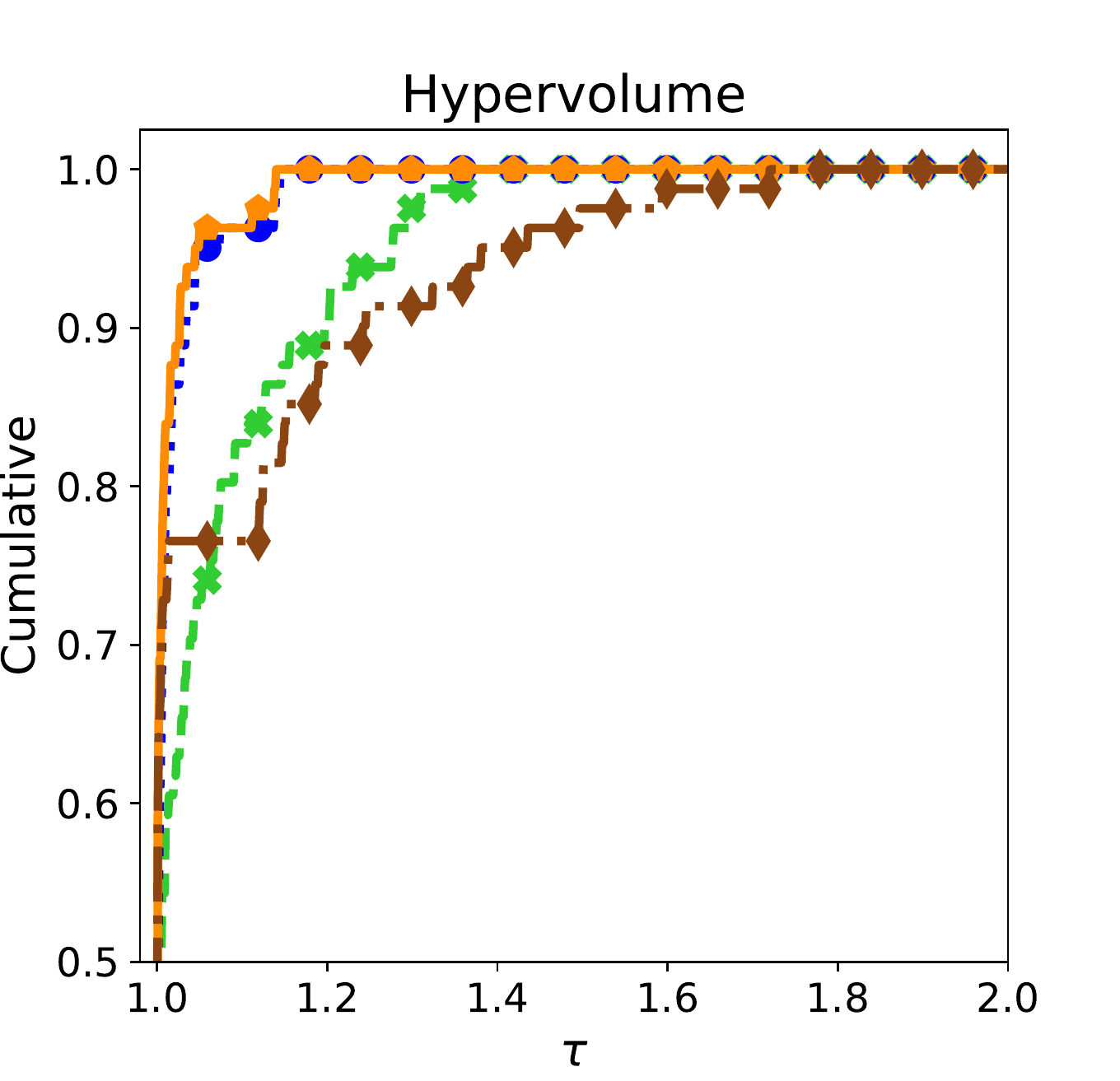}}
	\hfil
	\subfloat[]{\includegraphics[width=0.25\textwidth]{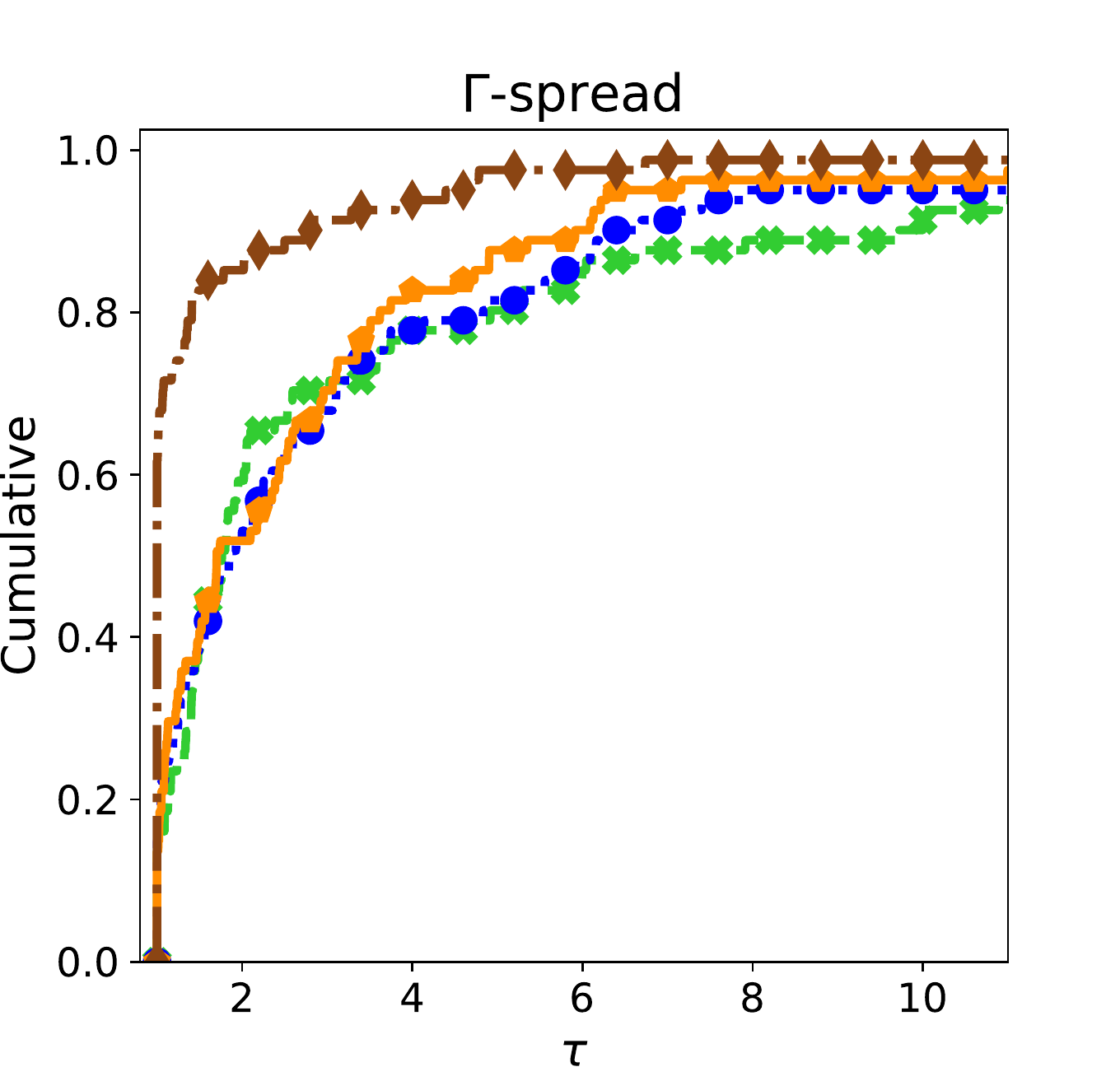}}
	\hfil
	\subfloat[]{\includegraphics[width=0.25\textwidth]{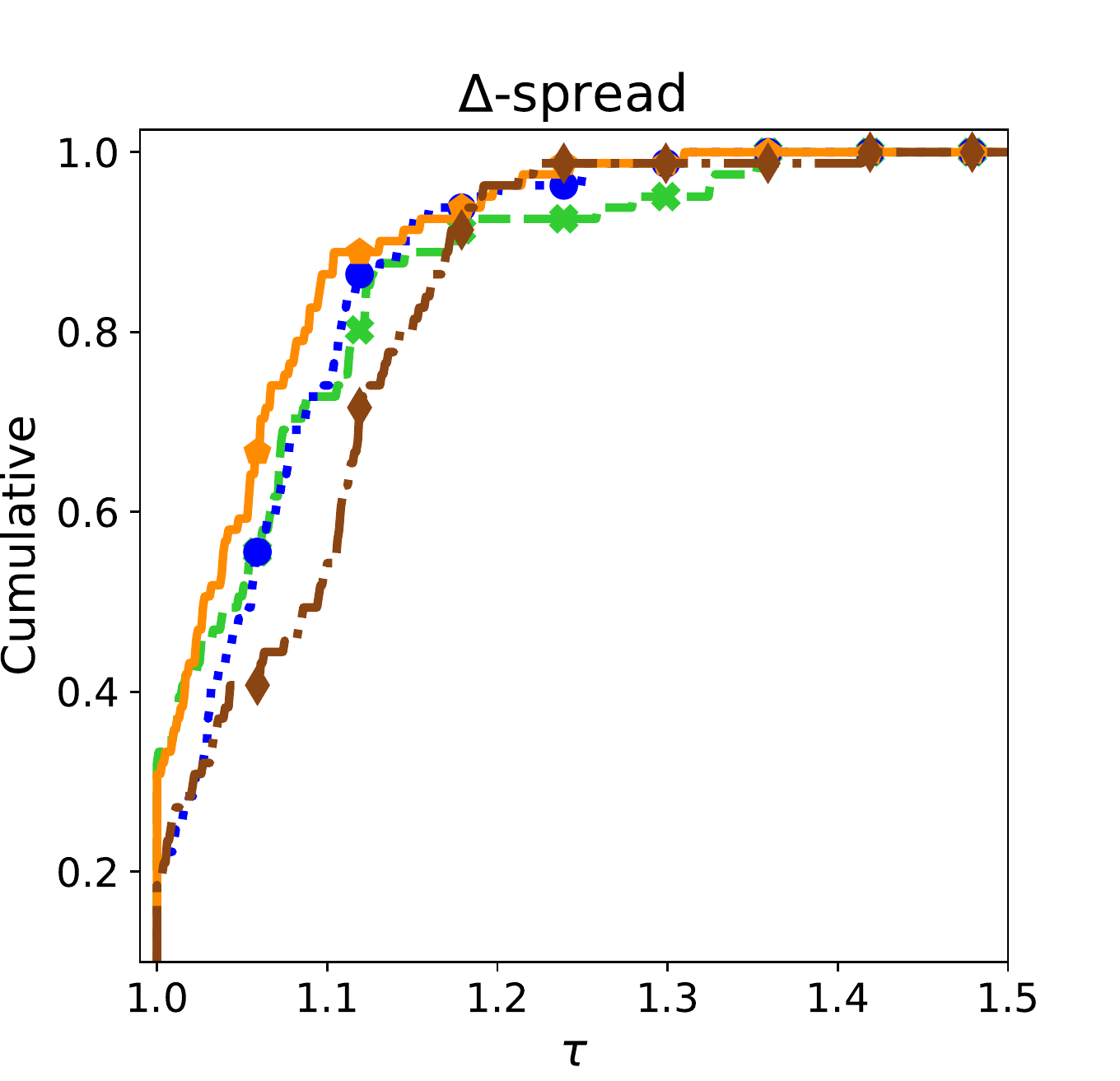}}
	\caption{Performance profiles for \texttt{SFSD} with different initialization strategies, i.e., \texttt{MOIHT}, \texttt{MOSPD}, \texttt{MOHyb} (best executions w.r.t.\ \textit{purity}) and \texttt{MIQP} on the quadratic problems.}
	\label{fig::QP_PP}
\end{figure*}

Looking at the \textit{purity} and the \textit{hyper-volume} metrics, \texttt{SFSD} resulted to be more robust with \texttt{MOHyb} as initialization strategy instead of \texttt{MOIHT} and \texttt{MOSPD}. These results reflect the behavior of the three single-point algorithms already shown in Figure \ref{fig::QP_2}: while \texttt{MOIHT} resulted to be more effective on the well-conditioned problems, \texttt{MOSPD}, with a right choice for the value for $\tau_0$, performed better on the (larger) set of ill-conditioned problems; \texttt{MOHyb}, inheriting the mechanisms of both, managed to obtain good results on problems of both types. As for $\Gamma$\textit{--spread}, \texttt{MIQP} proves to be more capable than the other single-point methods in generating solutions in the extreme regions of the objectives space, and this fact allowed \texttt{SFSD} to get wider front reconstructions. The performance of our front-oriented approach with \texttt{MOHyb}, \texttt{MOIHT} and \texttt{MOSPD} employed in the phase one was quite similar in this scenario. Regarding the $\Delta$\textit{--spread} metric, i.e., uniformity of the Pareto front approximation, all the initialization strategies led to comparable results.

\subsection{Logistic Regression}

In this last section, we analyze the performance profiles (Figure \ref{fig::LOG_PP}) on the logistic regression problems for \texttt{SFSD} with the different possible choices for the first phase of the algorithm. The values for the parameters of the algorithms were again chosen based on preliminary experiments which are not reported for the sake of brevity. In particular, we set: $L = 1.1\max\{L(f_1), L(f_2)\}$ for \texttt{MOIHT}; $\varepsilon = 10^{-7}$ for both \texttt{MOIHT} and \texttt{MOSD}; $\tau_0 = 1$, $\tau_{k + 1} = 1.3\tau_k$, $\varepsilon_0 = 10^{-5}$, $\varepsilon_{k + 1} = 0.9\varepsilon_k$ and $\|x_{k + 1} - y_{k + 1}\| \le 10^{-3}$ as stopping condition for \texttt{MOSPD}; $\alpha_0 = 1$, $\delta = 0.5$ and $\gamma = 10^{-4}$ for all Armijo-type line searches.
Again, the parameters for \texttt{MOIHT} and \texttt{MOSPD} were also employed in \texttt{MOHyb}.
Finally, since the objective functions have different scales, similarly to what is done in \cite{lapucci22}, when computing the \textit{spread} metrics we considered the logarithm (base 10) of the $f_2$ values and, then, re-scaled both $f_1$ and $f_2$ to have values in $[0, 1]$.

With respect to the \textit{purity} and the \textit{hyper-volume} metrics, \texttt{SFSD} resulted to be more robust with \texttt{MOIHT}, \texttt{MOSPD} and \texttt{MOHyb} as initialization strategies, with \texttt{MOHyb} appearing to be slightly superior. As for the $\Gamma$\textit{--spread} metric, \texttt{GSS} was the best algorithm for the \texttt{SFSD} phase one. However, although \texttt{SFSD}, equipped with this setting, was effective in reaching remote regions of the objectives space, it struggled to obtain uniform Pareto front approximations and, thus, to obtain good $\Delta$\textit{--spread} values. As for this last metric, using as initialization strategy \texttt{MOIHT} proved to be a better choice. 

\begin{figure*}
	\centering
	\subfloat[\label{fig::LOG_PP_Purity}]{\includegraphics[width=0.25\textwidth]{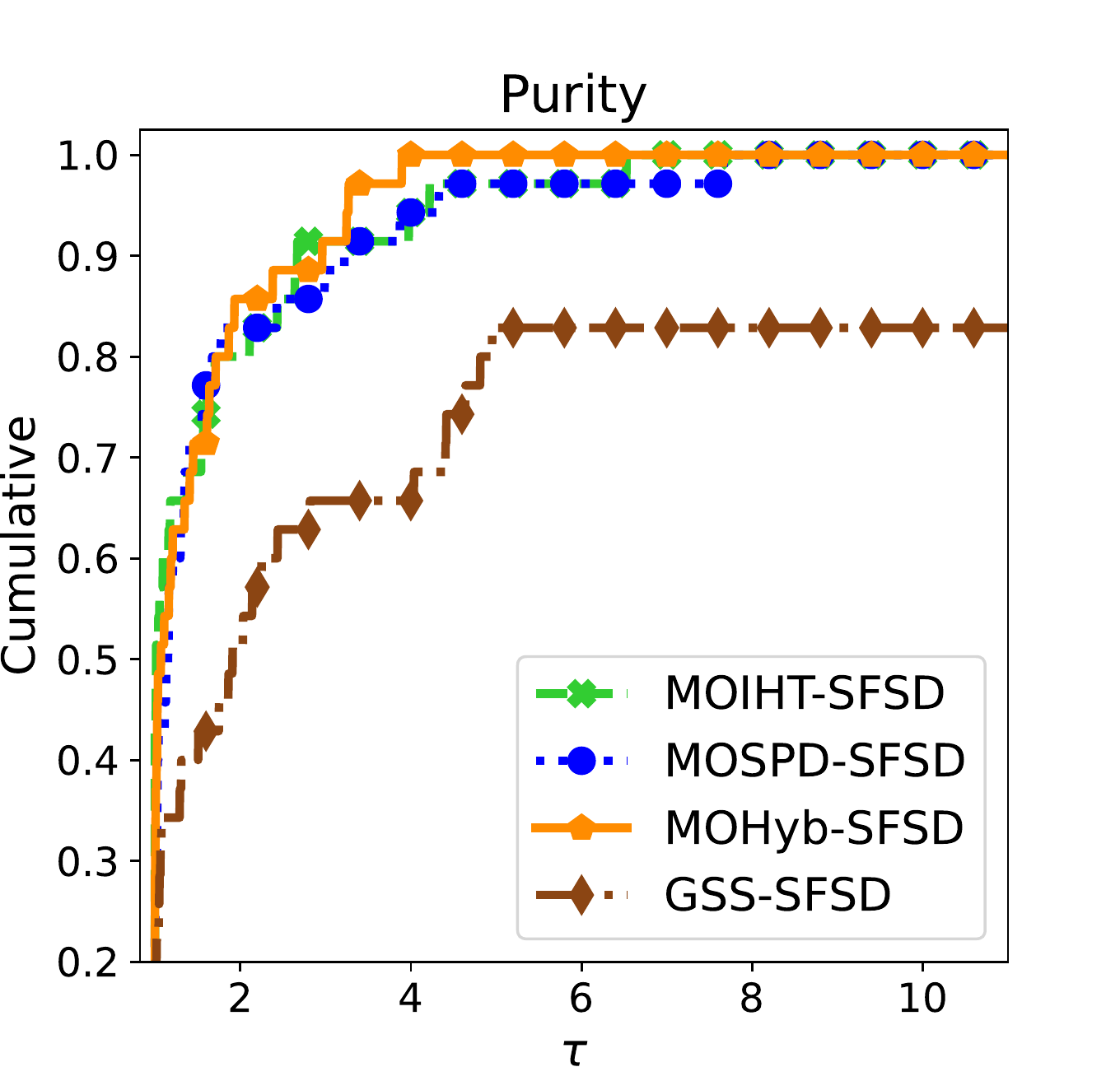}}
	\hfil
	\subfloat[]{\includegraphics[width=0.25\textwidth]{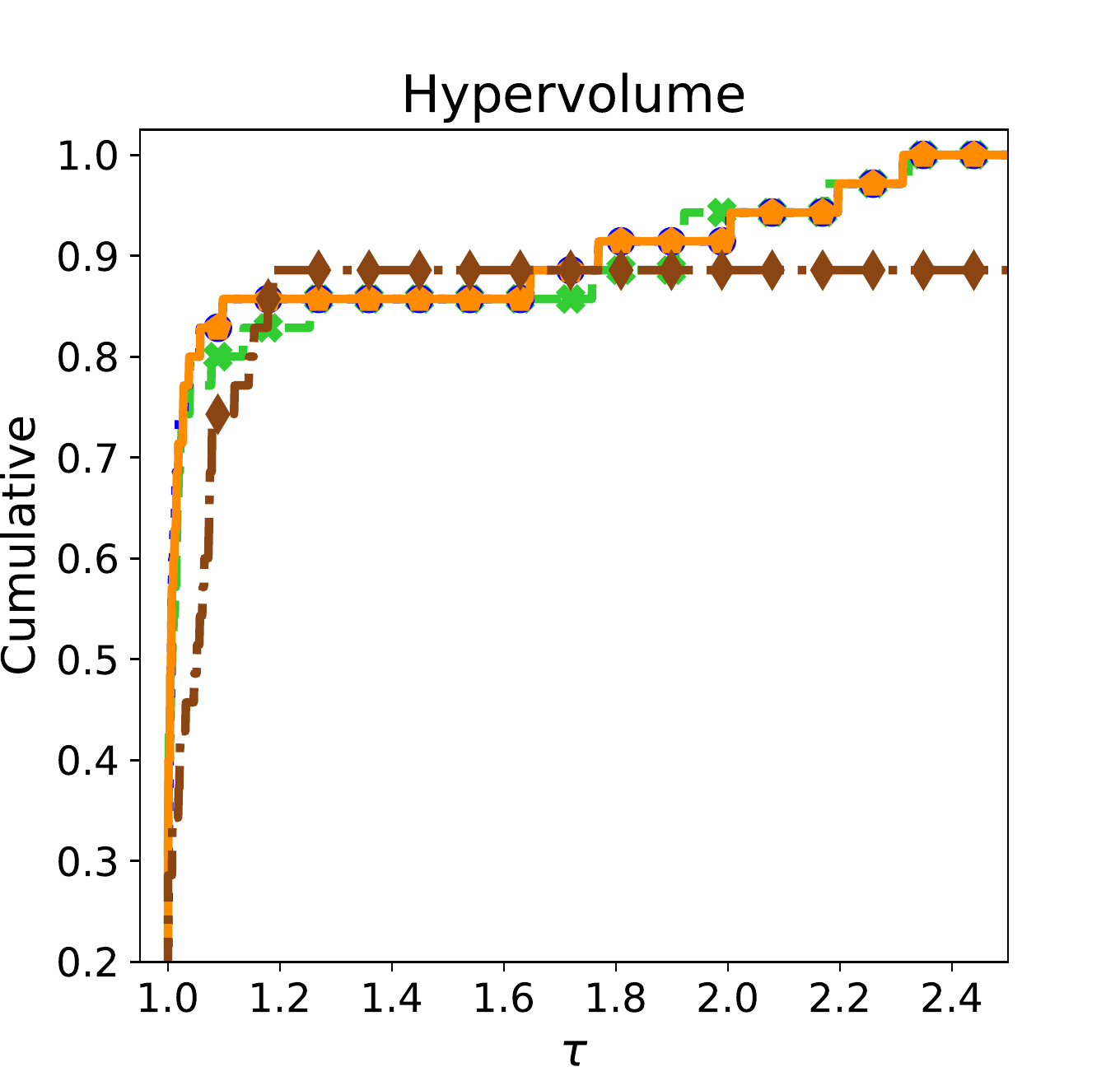}}
	\hfil
	\subfloat[]{\includegraphics[width=0.25\textwidth]{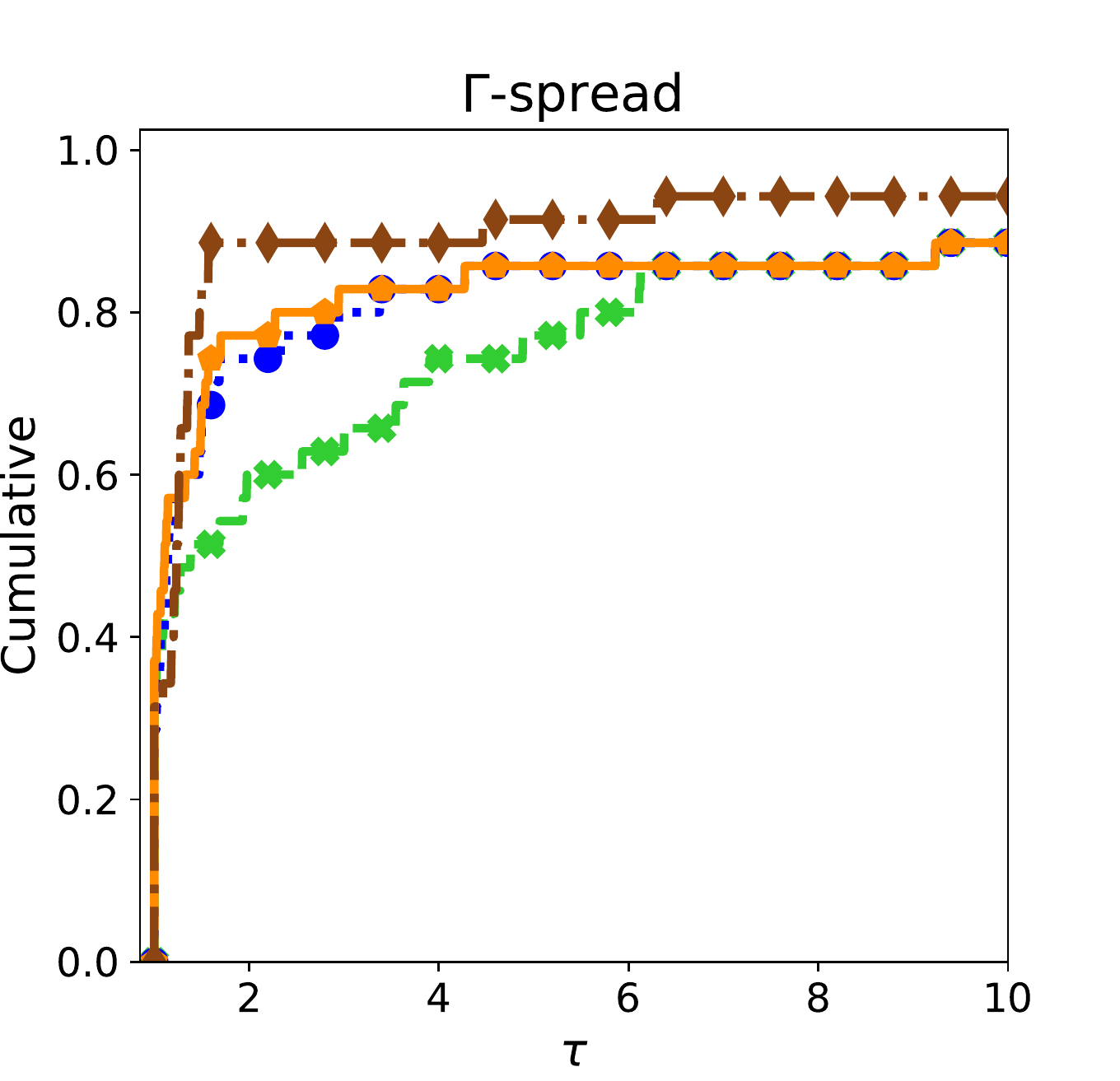}}
	\hfil
	\subfloat[]{\includegraphics[width=0.25\textwidth]{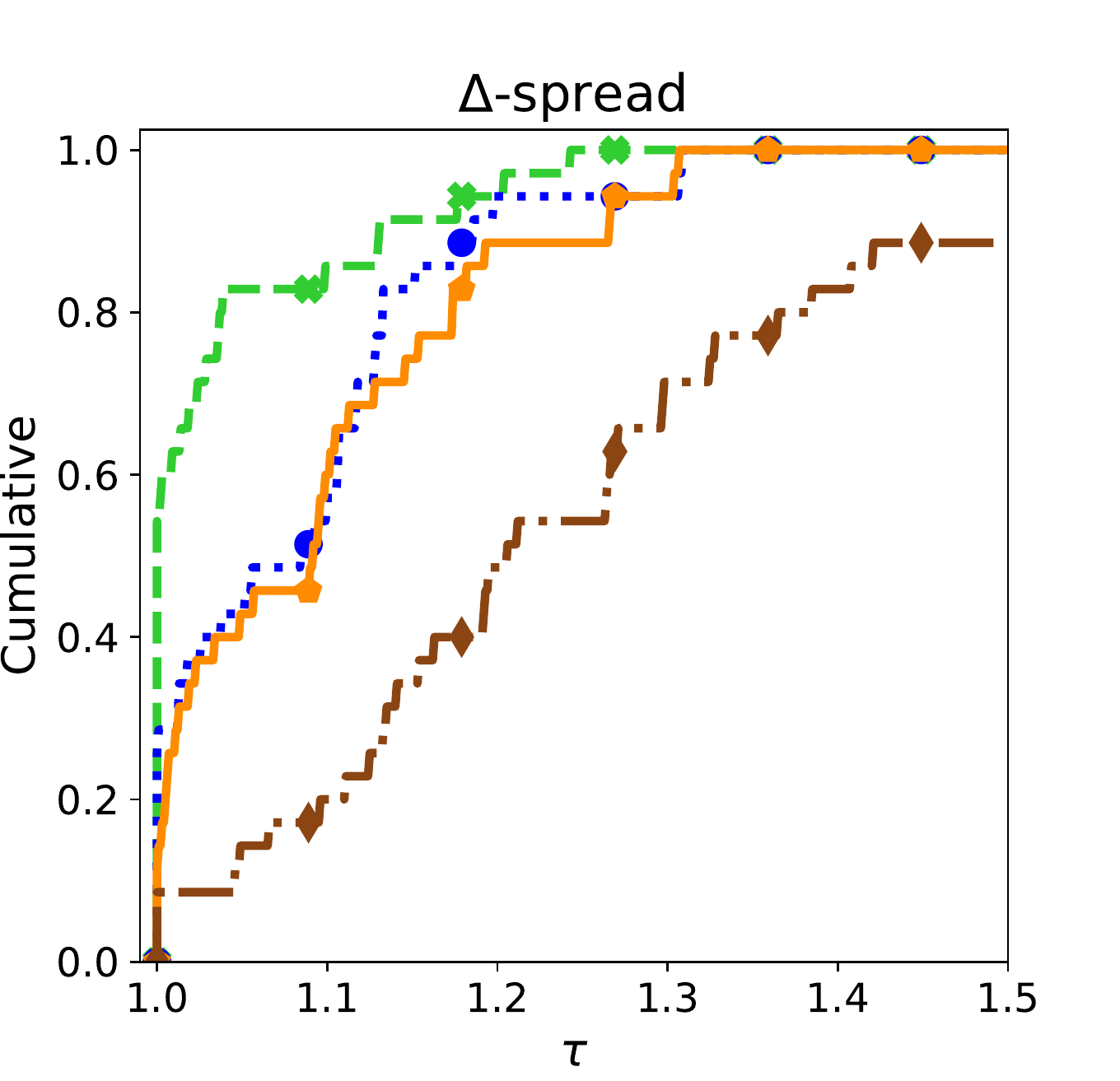}}
	\caption{Performance profiles for \texttt{SFSD} with different initialization strategies, i.e., \texttt{MOIHT}, \texttt{MOSPD}, \texttt{MOHyb} (best executions w.r.t.\ \textit{purity}) and \texttt{GSS} on 35 logistic regression problems.}
	\label{fig::LOG_PP}
\end{figure*}

\begin{remark}
	In the previous sections, we considered the best executions of \texttt{SFSD} equipped with \texttt{MOIHT}, \texttt{MOSPD} and \texttt{MOHyb}, and we compared them with the deterministic outputs of our front-oriented algorithm when \texttt{MIQP}/\texttt{GSS} was employed in the phase one. Thus, for the sake of completeness, in Figure \ref{fig::Worst_PP} we report the performance profiles w.r.t. the \textit{purity} metric obtained considering the worst runs. Comparing these performance profiles with the ones in Figures \ref{fig::QP_PP_Purity}-\ref{fig::LOG_PP_Purity}, we observe only slight decreases in the performance of the non-deterministic strategies.
\end{remark}

\begin{figure*}
	\centering
	\subfloat[]{\includegraphics[width=0.25\textwidth]{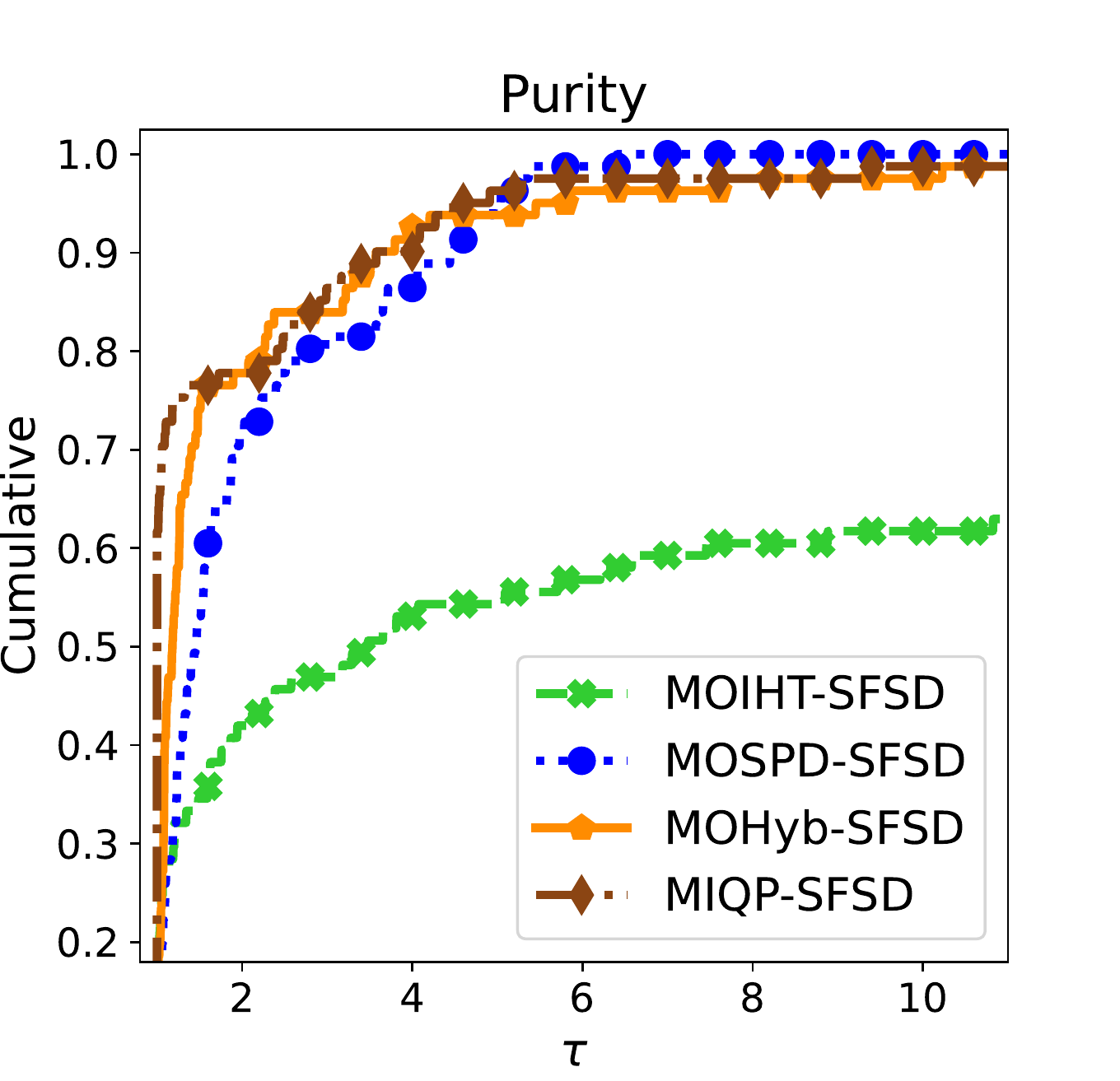}}
	\hfil
	\subfloat[]{\includegraphics[width=0.25\textwidth]{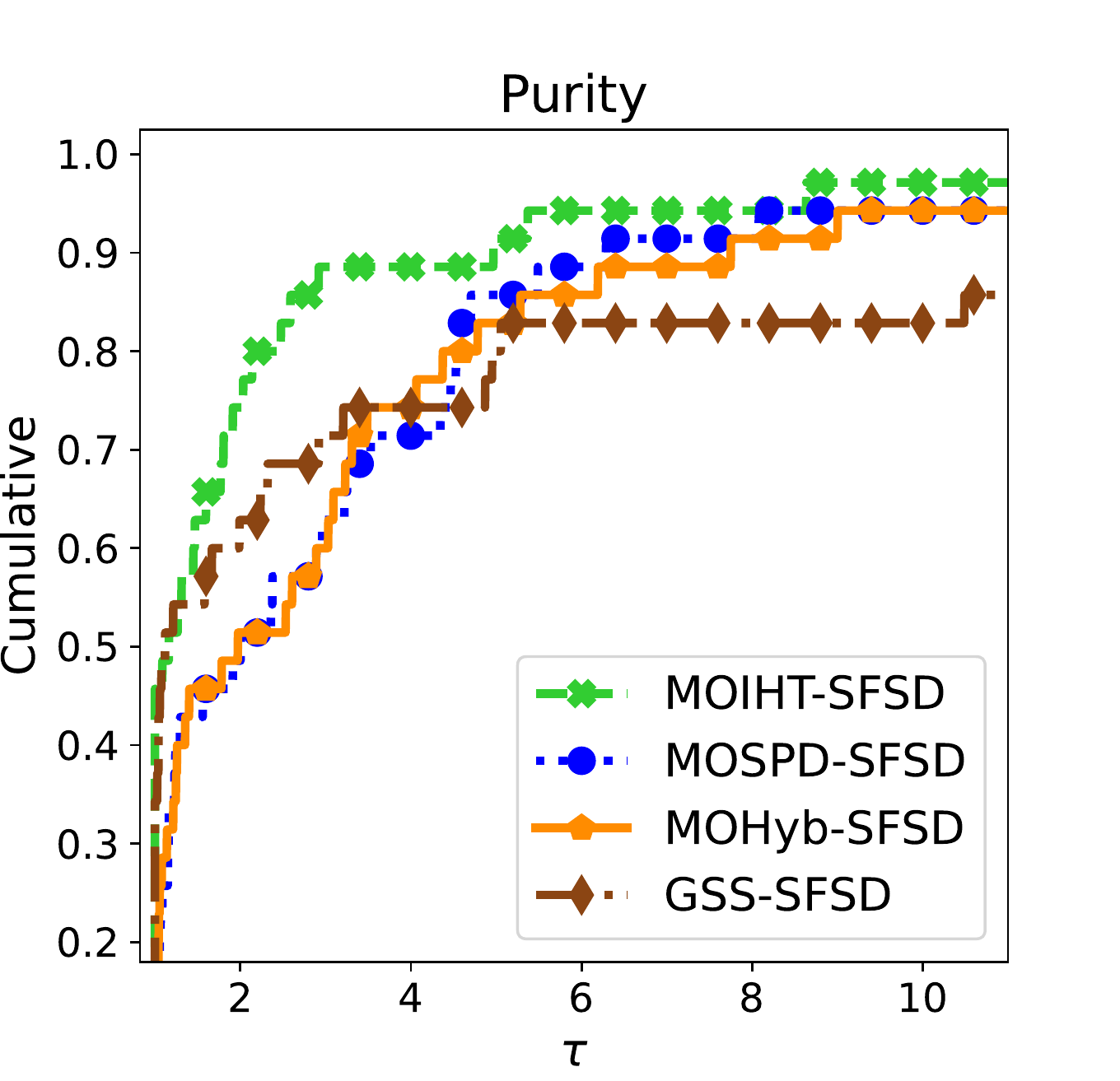}}
	\caption{Performance profiles for \texttt{SFSD} with different initialization strategies, i.e., \texttt{MOIHT}, \texttt{MOSPD}, \texttt{MOHyb} (worst executions w.r.t.\ \textit{purity}) and \texttt{MIQP}/\texttt{GSS}. (a) Quadratic problems; (b) Logistic regression problems.}
	\label{fig::Worst_PP}
\end{figure*}
\section{Conclusions}
\label{sec::conclusions}

In this paper, we considered cardinality-constrained multi-objective optimization problems. Inspired by the homonymous condition for sparse SOO \cite{beck13}, we defined the $L$-stationarity concept for MOO and we analyzed its relationships with the main Pareto optimality concepts and conditions. 

Then, we proposed two novel algorithms for the considered class of problems. The first one is an extension of the \textit{Iterative Hard Thresholding} method \cite{beck13} to the MOO case, called \texttt{MOIHT}: like the original approach, it aims to generate an $L$-stationary point. The second algorithm called \textit{Sparse Front Steepest Descent} (\texttt{SFSD}) is, to the best of our knowledge, the first front-oriented approach for cardinality-constrained MOO. Being an adaptation of the front steepest algorithm \cite{Cocchi2020_Onthe,lapucciimproved}, \texttt{SFSD} aims to approximate the (typically irregular and fragmented) Pareto front of the problem at hand. The method depends on suitable initialization strategies, including, e.g., multi-starting the \texttt{MOIHT} or the \texttt{MOSPD} \cite{lapucci22} algorithms, an hybridization of the two, or a scalarization approach.
From a theoretical point of view, we proved for \texttt{MOIHT} that the sequence of points converges to $L$-stationary solutions; for \texttt{SFSD}, on the other hand, we stated global convergence to points satisfying the MO Lu-Zhang optimality conditions. 

By a numerical experimentation, we also evaluated the performance of the proposed methodologies on benchmarks of quadratic and logistic regression problems. The \texttt{SFSD} methodology is thus shown to be successful at spanning the Pareto front in an exhaustive way, with the multi-start hybrid \texttt{MOSPD}-\texttt{MOIHT} procedure being the most promising solution to be used in the first phase of the algorithm. 

As for future works, we might extend the theoretical results to handle constraints other than the cardinality one. Moreover,
further researches might be focused on the performance evaluation of the algorithms on a more extensive set of problems with more general and possibly non-convex objective functions.

\renewcommand{\theequation}{A\arabic{equation}}
\setcounter{equation}{0}

\renewcommand{\thetable}{C\Roman{table}}
\setcounter{table}{0}
\renewcommand{\theHtable}{Supplement\thetable}

\renewcommand{\theproposition}{A\arabic{proposition}}
\setcounter{proposition}{0}
\renewcommand{\theassumption}{A\arabic{assumption}}
\setcounter{assumption}{0}

\renewcommand{\thealgocf}{A\arabic{algocf}}
\setcounter{algocf}{0}

{\appendices
	\section{Proof of Lemma \ref{lem::sufficientB}}
\label{app::proof_basic_results}
\begin{proof}
	Since $\bar{x}$ is a Pareto-stationarity point for problem \eqref{eq::mo-prob}, by Definition \ref{def::pareto_stat}, we have that $\theta(\bar{x}) = 0$, i.e.,
	\begin{equation}
		\label{eq::theta_B_unfold}
		\max_{j \in \left\{1,\ldots, m\right\}}\nabla f_j\left(\bar{x}\right)^\top d + \frac{1}{2} \left\| d \right\|^2 \ge 0, \qquad \forall d \in \mathcal{D}\left(\bar{x}\right).
	\end{equation}
	Let us suppose, by contradiction, that there exists a direction $\hat{d} \in \mathcal{D}(\bar{x})$ such that
	\begin{equation}
		\label{eq::negative_directional_derivative}
		\max_{j \in \left\{1,\ldots, m\right\}}\nabla f_j\left(\bar{x}\right)^\top\hat{d} < 0.
	\end{equation}
	Since \eqref{eq::theta_B_unfold} holds, we then deduce that $\frac{1}{2} \| \hat{d} \|^2 \ge | \max_{j \in \{1,\ldots, m\}}\nabla f_j(\bar{x})^\top\hat{d} \;| > 0.$
	Now, let us introduce the function $\tilde{\theta}: \mathbb{R}^n \times \mathbb{R}^n \times [0, 1] \rightarrow \mathbb{R}$ as $\tilde{\theta}(x, d, t) = \max_{j \in \{1,\ldots, m\}}\nabla f_j(x)^\top(td) + \frac{1}{2} \| td \|^2 = t\max_{j \in \{1,\ldots, m\}}\nabla f_j(x)^\top d + \frac{t^2}{2} \| d \|^2.$
	By \eqref{eq::basic_feasibility_MOO} and the feasibility of $\hat{d}$, it follows that, for all $t \in [0, 1]$, $t\hat{d} \in \mathcal{D}(\bar{x})$ and $\theta(\bar{x}) \le \tilde{\theta}(\bar{x}, \hat{d}, t)$.
	It is easy to see that $\tilde{\theta}(\bar{x}, \hat{d}, t) < 0$ if $0 < t < -({2}/{\|\hat{d}\|^2})\max_{j \in \{1,\ldots, m\}}\nabla f_j(\bar{x})^\top\hat{d}$,
	where the right-hand side is a positive quantity by Equation \eqref{eq::negative_directional_derivative}.
	But, in this case, we would have that $\theta(\bar{x}) \le \tilde{\theta}(\bar{x}, \hat{d}, t) < 0$, which contradicts the Pareto-stationarity of $\bar{x}$.
	Thus, we prove that, if $\bar{x}$ is Pareto stationary for problem \eqref{eq::mo-prob}, then $\max_{j \in \{1,\ldots, m\}}\nabla f_j(\bar{x})^\top d \ge 0, \; \forall d \in \mathcal{D}(\bar{x}).$
	From this point forward, we can follow the proof of Proposition 3.5 in \cite{lapucci22}.
\end{proof}

}

\section*{Conflict of Interest}
The authors declare that they have no conflict of interest.

\section*{Funding Sources}
This research did not receive any specific grant from funding agencies in the public, commercial, or not-for-profit sectors.

\section*{Data Availability Statement}
All the datasets analyzed during the current study are available in the  UCI Machine Learning Repository \cite{Dua2019}, \url{https://archive.ics.uci.edu/ml/datasets.php}.

\bibliographystyle{abbrv}

\end{document}